\documentclass[12pt,preprint,authoryear]{elsarticle}

\usepackage[a4paper,margin=1in]{geometry}
\usepackage{amssymb}
\usepackage{amsmath}
\usepackage{tkz-fct}
\usepackage{tikz}
\usepackage{bm}
\usepackage{pgfplots}
\usepackage{algorithm}
\usepackage{algorithmicx}
\usepackage{tabularx}
\pgfplotsset{compat=1.18}
\usepackage{comment}
\usepackage{xcolor}
\usepackage{pifont}
\usepackage{algpseudocode}
\usepackage{longtable}
\usepackage[short,c2]{optidef}
\usepackage{graphicx}
\usepackage{subcaption}
\usepackage{mathtools}
\usetikzlibrary{calc}
\usepackage{multicol}
\PassOptionsToPackage{dvipsnames,svgnames}{xcolor}
\usepackage{textcomp}
\usepackage{amsthm}

\newtheorem{proposition}{Proposition}

\usepackage{setspace}
\usepackage{multirow}

\usepackage{mathtools}
\newlength{\constrRHSwidth}
\usepackage{anyfontsize}
\usepackage{booktabs}
\usepackage[table]{xcolor}
\usepackage{tabularx}

\definecolor{fixedzero}{HTML}{F4CCCC}
\definecolor{fixedone}{HTML}{D9EAD3}
\definecolor{partial}{HTML}{E7E6E6}

\biboptions{authoryear,round,sort}
\journal{European Journal of Operational Research}

\begin{document}

%\pgfplotsset{compat=1.18} 

\begin{frontmatter}



\title{Largest Rashomon sets of decision trees for robust contextual optimization}


\author[dtumanagement]{Lorenzo Bonasera\corref{cor1}}
\ead{lobon@dtu.dk}

\author[dtumanagement]{David Pisinger}
\ead{dapi@dtu.dk} % replace/check if needed

\cortext[cor1]{Corresponding author}

\affiliation[dtumanagement]{
  organization={DTU Management, Technical University of Denmark},
  addressline={Akademivej, Building 358},
  city={Kongens Lyngby},
  postcode={2800},
  country={Denmark}
}

%% Abstract
\begin{abstract}
Many decision trees fit the same data almost equally well, yet they can route a query point to
different leaves and induce different local empirical distributions. We study decisions that meet
prescribed cost, shortage or risk targets despite this predictive multiplicity. We propose
the joint Rashomon and robustness optimization framework for optimal decision trees. It jointly
selects an operational decision and the largest Rashomon set of trees, so that the targets hold
under the local empirical distribution that every tree in this set induces at the query point.
We specialize the framework to the regression setting, and we show that a tree affects the
decision only through the training observations sharing the query leaf, which we call its query
neighborhood. As a result, the robust problem involves only finitely many distinct constraints,
which can be examined in order of increasing estimation loss. We develop a constraint generation
algorithm that combines query-path pricing with dynamic programming to identify violating
neighborhoods without enumerating trees. On synthetic newsvendor instances, the algorithm
typically needs few neighborhoods and runs substantially faster than full neighborhood
enumeration. On restaurant demand data, the robust orders increase the mean tolerated excess
estimation loss by 17.6\% and reduce the empirical conditional value-at-risk of the worst 10\% of
realized costs by 8.3\% relative to the sample average approximation orders of the optimal tree, while the mean
cost difference is not statistically significant. An interpretability analysis further shows how
the retained neighborhoods explain the decision and its robustness limit.
\end{abstract}



%% Keywords

\end{frontmatter}



\section{Introduction}\label{sec:intro}

During the past decade, data-driven prescriptive analytics has attracted considerable attention in operational research and management science. Many decision problems in these fields involve optimization under uncertainty, where conditional distributions of uncertain parameters can be estimated from historical observations and contextual covariates~\citep{bertsimas2020predictive, ban2019big}. This setting, often referred to as contextual stochastic optimization, creates a natural connection between machine learning and operational research by combining predictive models with prescriptive decision making, and it has motivated several methodological frameworks, including the widely adopted \textit{predict-then-optimize} (or \textit{estimate-then-optimize}) paradigm~\citep{sadana2025survey}. A fundamental issue for these frameworks is how well the predictive model generalizes beyond the training sample. Although bounds on estimation error can be derived under suitable assumptions, predictive performance may still be affected by model overfitting or distributional shifts in the data~\citep{quinonero2008dataset}. In contextual stochastic optimization, estimation error may propagate to the downstream decision problem and substantially affect solution quality, since the mapping from predictive to prescriptive solutions is often difficult to bound or characterize. This has motivated regret-based approaches that train predictive models for decision quality~\citep{elmachtoub2022smart}, and, in parallel, the use of robust and distributionally robust optimization to protect decisions against estimation error~\citep{kuhn2025distributionally}.

Within this literature, a central question is how much estimation uncertainty a decision can tolerate while meeting prescribed performance targets. \citet{zhu2022joint} propose joint estimation and robustness optimization (JERO), in which uncertainty is defined through sublevel sets of the estimation loss. The framework jointly selects a downstream decision and the largest loss tolerance for which all prescribed targets remain satisfied over the resulting uncertainty set. This principle is also relevant for tree-based models, which have demonstrated their usefulness in data-driven optimization~\citep{bertsimas2019optimal,elmachtoub2020decision}, with interpretability providing an important motivation for their use~\citep{rudin2019stop}. In contextual stochastic optimization, a decision tree can be used to construct a local empirical conditional distribution from the training observations assigned to the terminal leaf reached by the query point~\citep{bertsimas2020predictive}. Estimation uncertainty therefore concerns the choice of a tree partition and the local empirical distribution it induces. Applying JERO's paradigm in this setting requires accounting for the combinatorial structure of tree partitions and enforcing the prescribed targets over the distributions induced by near-optimal trees.
This view of estimation uncertainty is closely related to the \textit{Rashomon effect}~\citep{breiman2001statistical, xin2022exploring, paes2023inevitability} and to predictive multiplicity~\citep{black2022model}, whereby models with similar estimation losses yield conflicting predictions for individual observations. For trees, alternative near-optimal partitions may route the same query point to different terminal leaves and therefore induce different local empirical distributions. In light of the above, we extend the JERO framework to optimal decision trees (ODT) and refer to the resulting model as ODT-JRRO, where JRRO stands for \textit{joint Rashomon and robustness optimization}. The model jointly selects a downstream decision and an estimation-loss tolerance, seeking the largest Rashomon set for which the decision satisfies all prescribed targets under every induced distribution. We specialize our formulation to the regression setting. 

Our contributions are threefold. First, we formulate target-oriented
robustness over the empirical conditional distributions induced by
near-optimal regression trees, jointly selecting an operational decision
and its tolerance to excess estimation loss.
Second, we show that a tree affects the decision only through the training observations sharing the
query leaf, which we call its query neighborhood. We thereby reduce the
robust constraints to a finite family of query neighborhoods and
characterize the optimal tolerance through the minimum estimation loss at
which each neighborhood arises, distinguishing whether this tolerance is
attained. Third, we develop a constraint generation algorithm whose
separation procedure combines a decomposition along the path of the query
point with dynamic programming to identify violating neighborhoods without
enumerating the Rashomon set. Computational experiments on the contextual
newsvendor problem demonstrate the advantage of selective constraint
generation over full enumeration and identify nominal tree optimization as
a limitation at larger depths. A benchmark on real restaurant demand data shows greater robustness to alternative trees and lower conditional value-at-risk than the sample average approximation (SAA) on the nominal optimal tree. An interpretability case study
illustrates how the retained neighborhoods explain the change in the
decision and the point at which further robustness becomes infeasible.

\paragraph{Outline}
The remainder of this paper is organized as follows. Section~\ref{sec:related} reviews related work on contextual optimization, optimal decision trees, predictive multiplicity, and data-driven robust optimization. Section~\ref{sec:background} introduces the general problem setting and its specialization to regression trees. In Section~\ref{sec:methodology}, we reduce the robust constraints to a finite family indexed by query neighborhoods and characterize radius feasibility through their minimum-loss entry levels. Section~\ref{sec:dp} develops the query-path decomposition and presents the dynamic programming and constraint generation algorithms. Section~\ref{sec:computational} evaluates the proposed method using synthetic instances and a real-data benchmark for the contextual newsvendor problem. Finally, Section~\ref{sec:conclusion} summarizes our findings and outlines directions for future research.

\section{Related work}\label{sec:related}
Our work lies at the intersection of contextual prescriptive analytics, optimal decision trees, predictive multiplicity, and data-driven robust optimization. We review the relevant contributions and identify the modeling and computational differences that motivate our approach.

\subsection{Contextual optimization}

Contextual optimization combines learning from covariates with decision making under uncertainty. The interaction between prediction quality and downstream performance has motivated both decision-focused learning and data-driven approximations of conditional stochastic optimization problems. We refer to \citet{mandi2024decision} and \citet{sadana2025survey} for surveys of these approaches. One important stream of research modifies the learning problem by using loss functions that account for the downstream optimization problem, including the smart predict-then-optimize (SPO) approach of \citet{elmachtoub2022smart} and related end-to-end and decision-focused methods~\citep{donti2017task,wilder2019melding,el2023generalization,liu2021risk}.
A second stream, closer to our setting, incorporates covariate information into stochastic optimization to obtain data-driven prescriptions. \citet{deng2022predictive} develop predictive stochastic programming, integrating covariate-based learning with stochastic programming and studying model fidelity, statistical optimality, and generalization. \citet{bertsimas2020predictive} propose a weighted sample average approximation framework in which covariates determine the weights assigned to historical observations, providing local empirical approximations of the conditional distribution. Their framework includes weights induced by decision trees and random forests. \citet{kannan2025data} study sample average approximation frameworks that combine predictive models with residual-based scenario generation, including a construction using leave-one-out residuals, and establish consistency and finite-sample guarantees under suitable assumptions.
Tree-based methods construct contextual partitions that can be used to define local prescriptions. \citet{kallus2023stochastic} develop stochastic optimization forests, where tree splits are chosen to improve downstream decision quality. SPO Trees~\citep{elmachtoub2020decision} use the SPO loss to learn interpretable partitions, while \citet{bertsimas2019optimal} develop optimal prescriptive trees for personalized decision making. \citet{ribeiro2026integrated} integrate distribution estimation and prescription
through decision trees for contextual two-stage linear stochastic programs. Further applications of data-driven prescriptions address specific operational settings~\citep{ban2019big,ferreira2016analytics,liu2021time}. Our model uses the local weighting construction in which a regression tree assigns equal weight to observations sharing the query leaf. We consider the distributions induced by all tree partitions satisfying a common estimation-loss bound and seek a decision that meets the targets under every such distribution.

\subsection{Optimal decision trees and Rashomon sets}

Classical tree-induction algorithms construct trees through top-down greedy procedures, selecting at each branching node a split that optimizes a local criterion~\citep{breiman1984classification,c4.5}. 
Although computationally efficient, these procedures do not, in general, minimize a global tree-level objective. This limitation has motivated optimal decision tree methods, which jointly optimize the tree structure and its leaf predictions or labels.
The resulting learning problem is computationally challenging, with classical NP-hardness results for optimal binary decision trees~\citep{hyafil1976constructing}.
\citet{bertsimas2017optimal} propose a mixed-integer optimization framework for learning optimal classification trees of bounded depth.
In the regression setting, optimization-based formulations jointly determine the branching structure, the assignment of observations to leaves, and the numerical prediction at each leaf, often using a squared-error loss and a complexity penalty~\citep{dunn2018}.
Alternative formulations and computational approaches have been developed for optimal regression trees~\citep{blanquero2022sparse,zhang2023optimal,van2024piecewise} and optimal classification trees~\citep{verwer2017learning,blanquero2021optimal,categoricalodt,sparseodt,bonasera2024optimal,boutilier2023optimal,boutilier2022shattering,aghaei2025strong}.
Further extensions address specialized predictive and prescriptive settings~\citep{bertsimas2022survival,amram2022policy}.
We refer to \citet{carrizosa2021mathematical} for a survey of mathematical optimization approaches to tree induction.
A complementary line of work computes optimal trees by dynamic programming~\citep{aglin2020learning,lin2020generalized,demirovic2022murtree,van2023necessary}. For suitable losses and structural constraints, the subtrees below a split act on disjoint subsets of the data and can be optimized independently. In the regression setting, \citet{van2024piecewise} develop dynamic programming methods for piecewise constant and linear regression trees, including specialized algorithms for trees of depth two. Our method uses this decomposition to compute minimum-loss completions of query paths. For a fixed query path, the sibling subtrees are optimized independently, while the query leaf determines the local empirical distribution.

Even when an optimal predictive model is unique, many alternatives may attain losses close to the optimum. Following \citet{breiman2001statistical}, sets of such near-optimal models have been studied to assess variable importance, the existence of simpler models, and the inevitability of the effect itself~\citep{fisher2019all,semenova2022existence,paes2023inevitability}. The resulting disagreement on individual predictions, known as predictive multiplicity, has been measured and analyzed in classification~\citep{marx2020predictive,black2022model,hsu2022rashomon}, with implications for arbitrariness and trustworthiness in decision systems~\citep{creel2022algorithmic,hsu2024rashomongb}. For sparse classification trees, \citet{xin2022exploring} use dynamic programming to enumerate Rashomon sets. More recently, \citet{arslan2025sorted} develop SORTD, which enumerates trees in nondecreasing objective order and supports separable, totally ordered objectives, including regression loss. Our computational representation consists of the query neighborhoods induced by regression trees, each associated with the minimum estimation loss needed to realize it. Trees inducing the same query neighborhood generate identical downstream constraints. The separation procedure therefore searches for neighborhoods with minimum compatible tree loss on which the current decision violates a target, using optimal subtree computations to price their query paths. In contrast with enumeration approaches, this representation supports radius optimization and downstream certification without listing the trees in the Rashomon set.

\subsection{Data-driven robust optimization}
 
Robust optimization seeks solutions that remain feasible for every parameter realization in an uncertainty set \citep{soyster1973convex, ben1998robust}. The size of this set determines how conservative the solution is, and it is typically fixed in advance, for instance through the budget of uncertainty of \citet{bertsimas2004price}. Distributionally robust optimization applies the same worst-case principle to an ambiguity set of probability distributions \citep{wiesemann2014distributionally}. We refer to \citet{bertsimas2011theory}, \citet{rahimian2022frameworks}, and \citet{kuhn2025distributionally} for surveys. Data-driven uncertainty sets can be constructed using statistical information or learned representations~\citep{bertsimas2018data,goerigk2023data}, while distributional ambiguity sets include moment, $\phi$-divergence, Wasserstein, and relative-entropy formulations~\citep{delage2010distributionally,ben2013robust,mohajerin2018data,van2021data}.
More closely related to contextual decision making, \citet{chenreddy2022data} construct conditional uncertainty sets that vary with the observed covariates. \citet{kannan2024residuals} develop residuals-based distributionally robust formulations that combine predictive models with covariate information. In our setting, ambiguity arises from the local empirical distributions induced by all regression trees satisfying a common estimation-loss bound. Membership in the ambiguity family is therefore determined by the estimation loss of an inducing tree partition. For a fixed query point and loss tolerance, this family contains finitely many distinct empirical distributions. Since expectation is linear in the probability distribution, enforcing the expected-value targets over this family is equivalent to enforcing them over its convex hull. Thus, the fixed-radius problem also admits a distributionally robust interpretation, without changing the robust constraints or requiring mixtures to be induced by individual trees.

A related line of work evaluates decisions relative to prescribed performance targets. \citet{brown2009satisficing} introduce satisficing measures that evaluate risky positions through their ability to attain desired goals. \citet{long2023robust} develop robust satisficing, which minimizes a fragility measure describing how performance relative to a target deteriorates as the distribution deviates from an empirical reference. The JERO framework of \citet{zhu2022joint}, introduced in Section \ref{sec:intro}, is the direct methodological predecessor of our work. The authors derive explicit formulations for ordinary least squares, LASSO, and maximum likelihood estimation, as well as for uncertainty in probability distributions, including a finite-support model with estimated scenario probabilities. Under suitable convexity assumptions, robust counterparts are obtained using support functions and conjugate duality, and the radius is optimized by binary search over fixed-radius problems. They also develop second-order conic approximations for formulations involving exponential cones.
We retain JERO's loss-sublevel construction and joint optimization of the radius and decision. Our contribution concerns the combinatorial family of regression-tree partitions and the local empirical conditional distributions they induce at a given query point. We reduce the robust constraints to query neighborhoods, characterize their minimum-loss entry levels, and develop a constraint generation algorithm that searches over query paths using optimal subtree computations. The result is a solution method for the tree-based extension of JERO.

\section{Problem formulation}\label{sec:background}

We introduce a joint estimation and robustness optimization model in which the uncertainty set is defined as the Rashomon set of regression trees. The construction follows the JERO principle of defining uncertainty through sublevel sets of an estimation loss and selecting an operational decision that remains feasible for the largest admissible set. In contrast with classical robust optimization, the size of the uncertainty set is optimized jointly with the operational decision.

\subsection{Contextual decision problem}
We consider a contextual decision problem in which, after observing a covariate vector $x \in\mathcal X\subseteq\mathbb R^{d_x}$, the decision maker chooses a decision $z\in\mathcal Z\subseteq\mathbb R^{d_z}$ to minimize a cost function $c(z,y)$ that depends on an uncertain quantity $y\in\mathcal Y\subseteq\mathbb R^{d_y}$. The feasible set $\mathcal Z$ is assumed to be nonempty and compact. Whenever convexity of the downstream feasible sets or restricted master problems is invoked, we additionally assume that $\mathcal Z$ is convex. For simplicity, we restrict our attention to the scalar case $d_y=1$. The contextual stochastic optimization problem~\citep{sadana2025survey} associated with a generic covariate realization $x$ is given by
\begin{equation}\label{eq:cso}
(\text{CSO}) \quad
z^\star(x) \in
\arg \min_{z \in \mathcal Z}
\mathbb E_{\mathbb P_{Y\mid X=x}} \bigl[c(z,Y)\bigr],
\end{equation}
where $\mathbb P_{Y\mid X=x}$ denotes the unknown conditional distribution of $Y$ given $X=x$. We assume that the expected cost is well defined and
finite for every $z\in\mathcal Z$.
In a data-driven setting, the conditional distribution $\mathbb P_{Y\mid X=x}$ is estimated from historical data. Let \mbox{$\mathcal D_N={(x_i,y_i)}_{i=1}^N$} denote a training dataset consisting of $N$ observed covariate–response pairs in $\mathcal{X} \times \mathcal{Y}$. Let $\Theta_N$ denote a class of predictors constructed from $\mathcal D_N$. Each predictor $T\in\Theta_N$ is assumed to induce, for every $x\in\mathcal X$, an empirical conditional distribution $\widehat{\mathbb P}_{T,N}(\cdot\mid x)$ over $\mathcal Y$.
The corresponding data-driven contextual stochastic optimization problem is then
\begin{equation}
    (\text{CSO}_N) \quad
    z^\star_N(x;T) \in
    \arg \min_{z \in \mathcal Z}
    \mathbb E_{\widehat{\mathbb P}_{T,N}(\cdot\mid x)}
    \bigl[c(z,Y)\bigr].
\end{equation}
In the robust variant of \citet{zhu2022joint}, the uncertainty set is induced by an estimation procedure, in which parameters whose empirical estimation loss is within a tolerance $r$ of the optimal estimation loss are treated as plausible inputs to the downstream optimization problem. We extend the same logic to a class of predictors or empirical distributions in the target-oriented setting.
To this end, let $g_k(z,y)$, $k=0,\ldots,K$, denote performance functions that are measurable in $y$ and, for each fixed $y$, finite, continuous, and convex in $z$ on a convex domain containing $\mathcal Z$. We set $g_0(z,y)\equiv c(z,y)$ and let $\tau_k\in\mathbb R$ denote the corresponding target levels.
For a fixed predictor $T\in\Theta_N$ and query point~$x^0\in\mathcal X$, we define the induced value of target $k$ as
\begin{equation}\label{eq:Gdist}
G_k(z,T;x^0)
:=
\mathbb E_{\widehat{\mathbb P}_{T,N}(\cdot\mid x^0)}
\bigl[g_k(z,Y)\bigr].
\end{equation}
When $T$ is a point predictor, we interpret it as inducing a degenerate distribution concentrated at $T(x^0)$, so that $G_k(z,T;x^0)=g_k(z,T(x^0))$.
The targets in \eqref{eq:Gdist} are expected values of the specified performance functions. For example, with a scalar order quantity $q$, the choices
\[
\begin{aligned}
    g_0(q,y)&=c_{\mathrm{u}}(y-q)^++c_{\mathrm{o}}(q-y)^+,\\
    g_1(q,y)&=(y-q)^+,\\
    g_2(q,y)&=\bigl(g_0(q,y)-a\bigr)^+
\end{aligned}
\]
give targets on expected cost, expected shortage, and expected cost in excess of a fixed threshold $a\in\mathbb R$, respectively, with underage and overage costs $c_{\mathrm{u}},c_{\mathrm{o}}>0$. All three functions are continuous and convex in $q$. The stockout indicator $\mathbf{1}_{\{y>q\}}$ is not convex in $q$, so a target on stockout probability is not covered by the pointwise convexity assumption.
To model estimation uncertainty, let $\rho(T;\mathcal D_N)$
denote the estimation loss of predictor~$T$ on the training set
$\mathcal D_N$, and define the nominal optimal loss
\begin{equation}\label{eq:nominal-loss}
\rho^\star
=\min_{T\in\Theta_N}\rho(T;\mathcal D_N).
\end{equation}
For a robustness radius, or loss tolerance, $r\geq 0$, define the corresponding set of near-optimal predictors as
\begin{equation}
\mathcal R_N(r)
=\{T\in\Theta_N:\rho(T;\mathcal D_N) \leq \rho^\star+r\}.
\label{eq:generic-rashomon-set}
\end{equation}
Thus, larger values of $r$ correspond to larger sets of predictive models whose estimation loss is close to the nominal optimum. In machine learning, $\mathcal R_N(r)$ corresponds to the Rashomon set of a predictor class $\Theta_N$.
Finally, we define the \textit{joint Rashomon and robustness optimization} (JRRO) problem as
\begin{subequations}\label{eq:starting}
\begin{align}
  \textnormal{(JRRO)}\qquad
  \sup_{z,r}\quad & r \\
  \textnormal{s.t.}\quad
  & G_k(z,T;x^0) \leq \tau_k, \quad \forall T\in\mathcal R_N(r),\quad k=0,\ldots,K, \label{eq:robust-starting}\\
  & z\in\mathcal Z,\\
  & 0\leq r\leq\bar r,
\end{align}
\end{subequations}
where $\bar{r}$ is a valid upper bound. The JRRO problem seeks a decision $z$ that satisfies all target constraints for the largest possible set of near-optimal predictors, where the size of this set is measured by the loss tolerance $r$. In this generalized formulation, the uncertainty set consists of predictors whose estimation loss lies within $r$ of the best attainable loss, and each such predictor induces a conditional predictive distribution at the query point $x^0$. The structure of the resulting problem depends on the form of functions $g_k(z,y)$, the Rashomon set $\mathcal{R}_N(r)$, and the feasible set $\mathcal{Z}$.

\subsection{Rashomon set of regression trees}
 
We specialize the predictor class to axis-aligned regression trees. A regression tree $T$ partitions the feature space $\mathcal X$ with a rooted binary split structure. Let $\mathcal I(T)$ and $\mathcal L(T)$ denote its sets of internal split nodes and terminal leaves, respectively. Each internal node $v\in\mathcal I(T)$ applies the split condition $a_v^\top x \leq b_v$, where $a_v \in \{0,1\}^{d_x}$ is a one-hot feature selector and $b_v\in\mathbb R$ is a threshold value, and routes a covariate vector $x$ to one of its children according to $a_v^\top x \le b_v$ (left branch) or $a_v^\top x > b_v$ (right branch). Each terminal leaf $\ell \in \mathcal L(T)$ assigns a constant prediction $\mu_\ell \in \mathbb R$ to the points it contains.
Let $\Theta_N$ denote the class of feasible regression trees of maximum depth $D$ constructed from $\mathcal D_N$ and satisfying a minimum leaf size requirement $N_{\min}\in\{1,\ldots,N\}$.
We define the historical observations that fall in a given region $A\subseteq\mathcal X$ as
\begin{equation}
\mathcal N(A)
=
\{i\in[N]:x_i\in A\},
\label{eq:region-index-set}
\end{equation}
where $[N]=\{1,\ldots,N\}$. The minimum leaf size condition imposes $|\mathcal N(\ell)|\geq N_{\min}$, for every $\ell\in\mathcal L(T)$.
A regression tree $T$ induces an empirical conditional distribution by assigning equal probability to the outcomes associated with the terminal leaf reached by the covariate vector~\citep{bertsimas2020predictive}. For $x\in\mathcal X$, let $\ell_T(x)\in\mathcal L(T)$ denote the terminal leaf reached by $x$. For contextual optimization, the relevant object is the empirical conditional distribution induced by $T$
\begin{equation}
\widehat{\mathbb P}_{T,N}(\cdot\mid x)
=
\frac{1}{|\mathcal N(\ell_T(x))|}
\sum_{i\in \mathcal N(\ell_T(x))}
\delta_{y_i}(\cdot),
\label{eq:tree-induced-empirical-distribution}
\end{equation}
where $\delta_{y_i}$ denotes the Dirac measure concentrated at $y_i$. 
Using only the fitted leaf mean replaces the empirical distribution \eqref{eq:tree-induced-empirical-distribution} by the degenerate distribution concentrated at
\[
    \overline  y_T(x)
    =
    \frac{1}{|\mathcal N(\ell_T(x))|}
    \sum_{i\in \mathcal N(\ell_T(x))}
    y_i,
\]
hence evaluating the target $k$ as
$
    \overline G_k(z,T;x)
    =
    g_k\bigl(z,\overline  y_T(x)\bigr).
$
The two quantities $G_k(z,T;x)$ and $\overline G_k(z,T;x)$ coincide when $g_k(z,\cdot)$ is affine in the uncertain outcome $y$. If $g_k(z,\cdot)$ is convex for fixed $z$, Jensen's inequality implies
\(
    \overline G_k(z,T;x) \leq G_k(z,T;x),
\)
so the fitted leaf mean can underestimate the value of a convex downstream cost or risk function, which can lead to invalid robustness certificates.
We now define the estimation loss associated with a tree partition. Fix a tree $T\in\Theta_N$. The corresponding regularized squared-error loss is
\begin{equation}
\rho(T;\mathcal D_N)
=
\min_{\{\mu_\ell\}_{\ell\in\mathcal L(T)}}
\left\{ \frac{1}{N} 
\sum_{i=1}^N
\bigl(y_i-\mu_{\ell_T(x_i)}\bigr)^2
\right\}
+
\lambda \Omega(T) = 
\frac{1}{N}
\sum_{\ell\in\mathcal L(T)}
\operatorname{SSE}\left(\mathcal N(\ell)\right)
+
\lambda|\mathcal I(T)|,
\label{eq:tree-loss}
\end{equation}
where the complexity measure $\Omega(T)$ corresponds to the number of internal splits, and $\lambda\geq 0$ is a regularization parameter.
The optimal loss~$\rho^\star$ is obtained by solving \eqref{eq:nominal-loss} through mixed-integer optimization or dynamic programming.
Since within-leaf squared error cannot exceed the total sum of squared errors (SSE), we impose the following upper bound on the radius $r$
\begin{equation}
    \bar{r} = \frac{1}{N} \sum_{i = 1}^N (y_i - \bar{y})^2 + \lambda (2^D-1) - \rho^\star,
    \label{eq:radius-cap}
\end{equation}
so that $\mathcal{R}_N(\bar r) = \Theta_N$. The extension to vector-valued leaf predictions $\mu_\ell\in\mathbb R^{d_y}$ is conceptually straightforward~\citep{levatic2018semi}. When the uncertain outcome is discrete, an analogous construction can be obtained using optimal classification trees~\citep{bertsimas2017optimal} by replacing the squared-error term in \eqref{eq:tree-loss} with a suitable classification loss. 

\section{Neighborhood reformulation and radius optimization}\label{sec:methodology}
In this section, we reduce the robust constraints~\eqref{eq:robust-starting} to a finite family indexed by query neighborhoods. We associate each attainable neighborhood with its minimum compatible estimation loss and characterize radius feasibility through the resulting entry levels. We also provide a sufficient condition for satisfying the targets over all attainable neighborhoods.

\subsection{Reduction to query neighborhoods}\label{sec:projection}
Once the query covariate vector $x^0$ is observed, each tree partition $T\in\Theta_N$ induces a local empirical distribution by routing~$x^0$ through the tree and selecting the historical observations that fall in the same terminal leaf. Fix a partition \(T\in\Theta_N\) and a query point
\(x^0\in\mathcal X\). We define the \emph{query neighborhood}
\begin{equation}
  S_T(x^0)
  :=\mathcal N\bigl(\ell_T(x^0)\bigr),
  \label{eq:query-leaf-index-set}
\end{equation}
that is, the tree-induced set of historical observations that belong to the same leaf as the query point. Thus, both the empirical conditional distribution and every downstream target value depend on \(T\) only through the query neighborhood \(S_T(x^0)\), and we refer to the map $T\mapsto S_T(x^0)$ as the query projection. This is the tree analogue of evaluating a learned parametric model at a query point.  In a linear regression model of the form $y=\beta^\top x+\varepsilon$, the query point enters through the scalar evaluation $f_\beta(x^0)=\beta^\top x^0$ of the fitted conditional mean. In the tree model, it instead
selects a terminal leaf and hence a local empirical distribution over the historical outcomes indexed by \(S_T(x^0)\).
Substituting the empirical conditional distribution in \eqref{eq:tree-induced-empirical-distribution} into \eqref{eq:Gdist} gives
\begin{equation}
G_k(z,T;x^0)
=
\frac{1}{|S_T(x^0)|}
\sum_{i\in S_T(x^0)}
g_k(z,y_i).
\label{eq:local-empirical-target}
\end{equation}
For a nonempty set $S\subseteq[N]$, define the average residual of target $k$ incurred with decision $z$ on $S$ as
\begin{equation}
  H_k(z,S)
  :=
  \frac{1}{|S|}\sum_{i\in S}\bigl(g_k(z,y_i)-\tau_k\bigr),
  \label{eq:residual}
\end{equation}
so that the local target constraint \mbox{$G_k(z,T;x^0)\leq\tau_k$} can be rewritten in terms of $S_T(x^0)$ as
\begin{equation}
  H_k\bigl(z,S_T(x^0)\bigr)\leq 0 .
  \label{eq:tree-cut}
\end{equation}
Since each $g_k(\cdot,y_i)$ is continuous and convex, every residual $H_k(\cdot,S)$ is also continuous and convex on a convex domain containing $\mathcal Z$. Consequently, each robust cut \eqref{eq:tree-cut} defines a closed subset of $\mathcal Z$, which is convex whenever $\mathcal Z$ is convex.
For a given query point $x^0$, define the collection of
attainable query neighborhoods by
\begin{equation}
  \mathcal S(x^0)
  :=\{S_T(x^0):T\in\Theta_N\}.
\label{eq:attainable-masks}
\end{equation}
For $S\in\mathcal S(x^0)$, define the minimum compatible tree loss and excess loss as
\begin{align}
  \rho_S(x^0)
  &:=\min_{\substack{T\in\Theta_N:\ S_T(x^0)=S}}
       \rho(T;\mathcal D_N),
       \label{eq:neighborhood-loss}\\
  e_S(x^0)
  &:=\rho_S(x^0)-\rho^\star.
       \label{eq:neighborhood-excess}
\end{align}
The minimum in \eqref{eq:neighborhood-loss} is attained, since both $\rho(T;\mathcal D_N)$ and $S_T(x^0)$ depend on $T$ only through the partition it induces on the finite set $\{ x_i : i \in [N]\} \cup\{x^0\}$, and there are finitely many such partitions. The excess loss $0 \leq e_S(x^0) \leq \bar r$ is the
smallest tolerance at which $S$ becomes admissible, and we refer to it as
the \textit{entry level} of $S$. For $r\geq 0$, the near-optimal neighborhood set is
\begin{equation}
  \mathcal S(x^0,r)
  :=\{S\in\mathcal S(x^0):e_S(x^0) \leq r\}.
\label{eq:mask-rashomon}
\end{equation}

We remark that \(\mathcal S(x^0,r)\) is the image of the Rashomon set \(\mathcal R_N(r)\) of regression trees under the query projection. Hence, enforcing the robust cuts \eqref{eq:tree-cut} over $\mathcal S(x^0,r)$ is the same as enforcing the target constraints \eqref{eq:robust-starting} over $\mathcal R_N(r)$.
For a given query point $x^0$, equations \eqref{eq:tree-cut} and \eqref{eq:mask-rashomon} yield the following formulation of JRRO
\begin{subequations}\label{eq:direct-mask-jero}
\begin{align}
  \textnormal{(ODT-JRRO)}\qquad
  \sup_{z,r}\quad & r \label{eq:jrro-obj}\\
  \textnormal{s.t.}\quad
  & H_k(z,S)\leq 0, \quad \forall S\in\mathcal S(x^0,r),\quad k=0,\ldots,K, \label{eq:masks}\\
  & z\in\mathcal Z, \label{eq:jrro-z}\\
  & 0\leq r\leq\bar r. \label{eq:jrro-r}
\end{align}
\end{subequations}
For fixed $r$, formulation~\eqref{eq:direct-mask-jero} is a feasibility problem over $\mathcal Z$ with finitely many neighborhood constraints, and is convex whenever $\mathcal Z$ is convex. The constraint family \eqref{eq:masks} changes only at the finitely many neighborhood entry
levels, and a neighborhood is already attainable at its own entry level. If feasibility first fails at a positive entry level, that level is excluded and every smaller nonnegative radius is feasible. Since the cap in \eqref{eq:radius-cap} includes all attainable neighborhoods, the cap is attained exactly when the full family of neighborhood constraints is feasible. Infeasibility at radius zero is treated separately below. The following subsection characterizes these cases through the finitely many entry levels.

\subsection{Radius optimization over entry levels}\label{sec:optmodel}
For a decision $z\in\mathcal Z$, define the first loss threshold at which
some target fails, that is
\begin{equation}
  r_{\mathrm{brk}}(z;x^0)
  :=
  \min
  \left\{
    e_S(x^0):
    S\in\mathcal S(x^0),\
    H_k(z,S)>0
    \text{ for some }k
  \right\},
  \label{eq:decision-breakdown}
\end{equation}
with $r_{\mathrm{brk}}(z;x^0)=+\infty$ if no attainable neighborhood
violates any target. We define the master problem as
\begin{equation}
  r^\star
  :=
  \min\Bigl\{\,\bar r,\ \max_{z\in\mathcal Z} r_{\mathrm{brk}}(z;x^0)\Bigr\}.
  \label{eq:optimal-breakdown}
\end{equation}
When $\mathcal Z$ is compact and every $H_k(\cdot,S)$ is continuous, the
inner maximum is attained and $r^\star$ coincides with the supremum of
ODT-JRRO whenever the latter is feasible.
We note that the set \mbox{$\{e_S(x^0):S\in\mathcal S(x^0)\}$} is finite, and
$\mathcal S(x^0,r)$ changes only when $r$ crosses one of its elements. Two
radii lying between consecutive entry levels induce identical constraint
families in \eqref{eq:masks}, so that optimizing $r$ requires examining only the finitely many attainable levels. Therefore, let
$0=\eta_1<\cdots<\eta_P$ denote the distinct entry levels, and set
$\eta_{P+1}:=+\infty$. By \eqref{eq:mask-rashomon}, the near-optimal
families attainable at these levels are exactly the sets
$\mathcal S(x^0,\eta_p)$, for $p\in[P]$. We accordingly define the feasible
set at level $\eta_p$ as
\begin{equation}
  \mathcal F_p:=\bigl\{z\in\mathcal Z: H_k(z,S)\leq0,\ \forall S\in\mathcal S(x^0,\eta_p),\ \forall k\bigr\}.
  \label{eq:prefix}
\end{equation}

\begin{proposition}\label{prop:closed-form}
The near-optimal families satisfy $\mathcal S(x^0,\eta_1)\subset\cdots\subset\mathcal S(x^0,\eta_P)=\mathcal S(x^0)$, and $\mathcal F_1\supseteq\cdots\supseteq\mathcal F_P$, each $\mathcal F_p$ being convex whenever $\mathcal Z$ is. There exists an index $p^\star\in\{0,\ldots,P\}$ such
that
\(
\{p\in[P]:\mathcal F_p\neq\varnothing\}
=
\{1,\ldots,p^\star\},
\)
with the convention $\{1,\ldots,0\}:=\varnothing$. Then
\begin{enumerate}
    \item If $p^\star=0$, then
    $r_{\mathrm{brk}}(z;x^0)=0$ for every $z\in\mathcal Z$,
    and ODT-JRRO is infeasible.

    \item If $1\le p^\star<P$, then every
    $z^\star\in\mathcal F_{p^\star}$ attains
    \(
    r_{\mathrm{brk}}(z^\star;x^0)
    =
    \eta_{p^\star+1}.
    \)
    The feasible radii of ODT-JRRO are
    $[0,\eta_{p^\star+1})$,
    and $r^\star = \eta_{p^\star+1}$.

    \item If $p^\star=P$, then every
    $z^\star\in\mathcal F_P$ satisfies all attainable
    neighborhoods and
    \(
    r_{\mathrm{brk}}(z^\star;x^0)=+\infty.
    \)
    In this case, $(z^\star,\bar r)$ is feasible for ODT-JRRO, which therefore attains
    the cap $r^\star =\bar r$.
\end{enumerate}
\end{proposition}

\begin{proof}
By \eqref{eq:mask-rashomon}, the near-optimal neighborhood families
are nested. The inclusions between successive families are strict,
since each distinct entry level admits at least one new neighborhood,
and $\mathcal S(x^0,\eta_P)=\mathcal S(x^0)$. Enlarging the neighborhood
family adds constraints, so $\mathcal F_{p+1}\subseteq\mathcal F_p$.
Each $\mathcal F_p$ is convex whenever $\mathcal Z$ is convex, because
the residuals $H_k(\cdot,S)$ are convex. Nestedness also implies that
the nonempty sets form an initial segment, which defines $p^\star$
as stated.
Every feasible tree has within-leaf squared error at most
$\sum_{i=1}^N(y_i-\bar y)^2$ and at most $2^D-1$ internal splits.
Thus the cap in \eqref{eq:radius-cap} satisfies $\eta_P\le\bar r$.
Since $\eta_1=0$ and $\eta_{P+1}=+\infty$, the intervals
$[\eta_p,\eta_{p+1})$, $p\in[P]$, partition $[0,+\infty)$.
The neighborhood family is constant on each such interval:
$\mathcal S(x^0,r)=\mathcal S(x^0,\eta_p)$ whenever
$r\in[\eta_p,\eta_{p+1})$. Consequently, a radius in this interval
is feasible for ODT-JRRO if and only if
$\mathcal F_p\neq\varnothing$ and $r\le\bar r$.
If $p^\star=0$, then $\mathcal F_1=\varnothing$, so every decision
violates at least one constraint entering at level zero.
Hence every decision has breakdown threshold zero, and ODT-JRRO
is infeasible.
If $1\le p^\star<P$, every $z\in\mathcal F_{p^\star}$ satisfies
all constraints entering strictly before $\eta_{p^\star+1}$.
Since $\mathcal F_{p^\star+1}=\varnothing$, each such decision
violates at least one constraint entering at $\eta_{p^\star+1}$
and therefore has breakdown threshold exactly $\eta_{p^\star+1}$.
Because $\eta_{p^\star+1}\le\bar r$, the feasible radii are
exactly $[0,\eta_{p^\star+1})$. Their supremum is therefore
$r^\star=\eta_{p^\star+1}$ and is not attained.
Finally, if $p^\star=P$, every $z\in\mathcal F_P$ satisfies
all attainable neighborhood constraints and has
$r_{\mathrm{brk}}(z;x^0)=+\infty$. Thus $(z,\bar r)$ is feasible
for every such decision, and $r^\star=\bar r$ is attained.
\end{proof}

Proposition \ref{prop:closed-form} reduces ODT-JRRO to the identification of a single
integer $p^\star$ by a monotone scan over the finitely many entry levels.
Since the near-optimal families are nested, feasibility is monotone in $p$,
and each test is a feasibility problem over the intersection of
$\mathcal Z$ with finitely many convex cuts, hence a convex feasibility
problem whenever $\mathcal Z$ is convex. In the second case, the optimal
radius is an excluded endpoint. Every tree with excess loss strictly below
$\eta_{p^\star+1}$ is certified, while no decision is feasible at the
endpoint itself.
Since every point of $\mathcal F_{p^\star}$ has breakdown threshold $r^\star$, we resolve the tie by selecting the decision with the lowest nominal cost under the distribution induced by a nominal optimal tree $T^0$
\begin{equation}
  z^\star\in\arg\min_{z\in\mathcal F_{p^\star}}\ H_0\bigl(z,S_{T^0}(x^0)\bigr).
  \label{eq:lexicographic}
\end{equation}
The nominal tree is selected by a fixed rule if the optimum is not unique and is kept unchanged throughout the algorithm. Figure~\ref{fig:scan} illustrates the monotone scan over entry levels and the nested structure of feasible sets.
\begin{figure}[t]
\centering
\begin{tikzpicture}[
  x=1.2cm,y=0.6cm,
  lbl/.style={font=\small},
  slbl/.style={font=\footnotesize},
  tick/.style={draw=black!55,line width=0.4pt,dashed,dash pattern=on 2pt off 2pt}
]
\def\ea{0} \def\eb{2.2} \def\ec{4.0} \def\ed{6.3} \def\ee{8.2}
\def\zs{3.20}

% nested feasible sets, one band per entry level
\fill[gray!14] (\ea,0.7) rectangle (\eb,5.3);
\fill[gray!26] (\eb,1.4) rectangle (\ec,5.3);
\fill[gray!40] (\ec,1.4) rectangle (\ed,4.2);
\fill[gray!56] (\ed,2.7) rectangle (\ee,4.2);

% staircase boundary
\draw[line width=0.9pt] (\ea,5.3) -- (\ec,5.3) -- (\ec,4.2) -- (\ee,4.2);
\draw[line width=0.9pt] (\ea,0.7) -- (\eb,0.7) -- (\eb,1.4) -- (\ed,1.4) -- (\ed,2.7) -- (\ee,2.7);

\foreach \x in {\eb,\ec,\ed} {\draw[tick] (\x,0) -- (\x,5.7);}
\draw[line width=1.1pt] (\ee,0) -- (\ee,5.7);

\draw[-{Latex[length=2mm]},line width=0.6pt] (0,0) -- (10.5,0) node[right,lbl] {$r$};
\draw[-{Latex[length=2mm]},line width=0.6pt] (0,0) -- (0,6.1) node[above,lbl] {$z$};

% band labels, pinned just inside the upper edge of each band
\node[slbl,anchor=north,yshift=-0.6mm] at ({(\ea+\eb)/2},5) {$\mathcal F_1$};
\node[slbl,anchor=north,yshift=-0.6mm] at ({(\eb+\ec)/2},5) {$\mathcal F_2$};
\node[slbl,anchor=north,yshift=-0.6mm] at ({(\ec+\ed)/2},4.2) {$\mathcal F_3$};
\node[slbl,anchor=north,yshift=-0.6mm] at ({(\ed+\ee)/2},4.2) {$\mathcal F_{p^\star}$};
\node[slbl,anchor=west,xshift=1.5mm] at (\ee,3.55) {$\mathcal F_{p^\star+1}=\varnothing$};

% neighborhood entering at each level
\node[slbl,anchor=south] at (0.5,5.75) {$S_1$};
\node[slbl,anchor=south] at (\eb,5.75) {$S_2$};
\node[slbl,anchor=south] at (\ec,5.75) {$S_3$};
\node[slbl,anchor=south] at (\ed,5.75) {$S_4$};
\node[slbl,anchor=south] at (\ee,5.75) {$S_5$};

\foreach \x/\t in {\ea/{$\eta_1=0$},\eb/{$\eta_2$},\ec/{$\eta_3$},\ed/{$\eta_4$}}
  {\draw (\x,0) -- ++(0,-1.5mm);
   \node[slbl,anchor=north,yshift=-2mm] at (\x,0) {\t};}
\draw (\ee,0) -- ++(0,-1.5mm);
\node[slbl,anchor=north,yshift=-2mm] at (\ee,0) {$\eta_5=r^\star$};

% certified radii
\draw[line width=1.6pt] (0,0) ++(0,-10mm) -- ++({\ee*1.2cm},0);
\fill (0,0) ++(0,-10mm) circle (1.7pt);
\draw[fill=white,line width=0.9pt] (\ee,0) ++(0,-10mm) circle (1.9pt);
\node[slbl,anchor=north,yshift=-11.5mm] at ({\ee/2},0) {certified radii $[0,r^\star)$};

% tie-break decision
\draw[dash pattern=on 2pt off 1.5pt,line width=0.6pt] (0,\zs) -- ({(\ed+\ee)/2},\zs);
\fill ({(\ed+\ee)/2},\zs) circle (2pt);
\node[slbl,anchor=east,xshift=-1mm] at (0,\zs) {$z^\star$};
\end{tikzpicture}

\caption{The monotone scan of Proposition \ref{prop:closed-form}, drawn for a
scalar decision. The neighborhood $S_5$ empties the feasible set at $\eta_5$, which is the supremum radius $r^\star$, and the decision $z^\star$ is selected
from $\mathcal F_{p^\star}$ by tie-breaking.}
\label{fig:scan}
\end{figure}
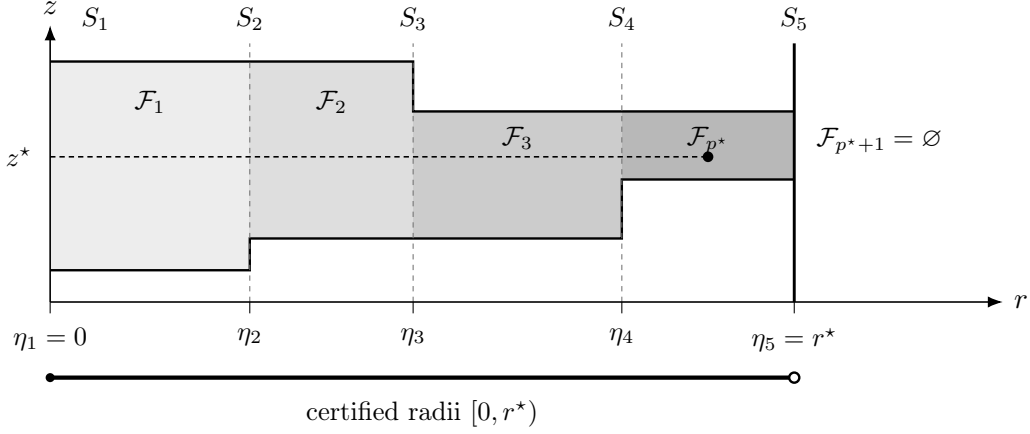

\subsection{Minimum leaf size and worst-case target values}
\label{sec:min-leaf-envelope}
The minimum leaf size limits how strongly a query
neighborhood can concentrate on observations with large
target values. To examine its role in worst-case target
evaluation, we relax tree attainability and consider all
subsets containing at least $N_{\min}$ observations.
This relaxation yields a conditional value-at-risk (CVaR) upper envelope under the
unconditional empirical distribution
\[
    \widehat{\mathbb P}_N
    :=\frac{1}{N}\sum_{i=1}^N\delta_{y_i}.
\]
The following proposition provides a reference for assessing
the additional conservatism introduced by disregarding
the restrictions imposed by the tree class.

\begin{proposition}\label{prop:cvar}
Fix $z\in\mathcal Z$ and a target $k$. Then
\begin{equation}
    \max_{S\in\mathcal S(x^0)}H_k(z,S)
    \leq
    \operatorname{CVaR}_{N_{\min}/N}
    \bigl[g_k(z,Y)\bigr]-\tau_k,
    \label{eq:cvar-envelope}
\end{equation}
where CVaR is evaluated under $\widehat{\mathbb P}_N$,
with $N_{\min}/N$ denoting the upper-tail fraction.
\end{proposition}

\begin{proof}
Every leaf receives at least $N_{\min}$ observations, so
$\mathcal S(x^0)\subseteq
\mathcal U:=\{S\subseteq[N]:|S|\geq N_{\min}\}$.
Write $w_i=g_k(z,y_i)-\tau_k$ and order the residuals as
$w_{(1)}\geq\cdots\geq w_{(N)}$.
Among index sets of fixed size $s$, the average $H_k(z,S)$
is maximized by the $s$ largest residuals, with value
$A(s)=\frac1s\sum_{j=1}^s w_{(j)}$.
For $1\leq s<N$, we have $w_{(s+1)}\leq A(s)$, hence
\(
  A(s+1)=\frac{sA(s)+w_{(s+1)}}{s+1}\leq A(s).
\)
Thus $A$ is nonincreasing and
$\max_{S\in\mathcal U}H_k(z,S)=A(N_{\min})$.
For a finite-support random variable $V$ and an upper-tail fraction
$\varepsilon\in(0,1]$, we recall
\begin{equation}
  \operatorname{CVaR}_{\varepsilon}(V)
  =\min_{t\in\mathbb R}
  \Bigl\{t+\tfrac1\varepsilon\,
  \mathbb E\bigl[(V-t)^+\bigr]\Bigr\},
  \label{eq:ru}
\end{equation}
which is the mean of the worst $\varepsilon$-fraction of
outcomes~\citep{rockafellar2000optimization}.
Under $\widehat{\mathbb P}_N$, with $\varepsilon=N_{\min}/N$
and $V=g_k(z,Y)-\tau_k$, the objective in~\eqref{eq:ru} is
\(
  F(t)=t+\frac1{N_{\min}}\sum_{i=1}^N(w_i-t)^+.
\)
Discarding the nonnegative terms beyond the first $N_{\min}$
ordered residuals and using $a^+\geq a$ gives, for every
$t\in\mathbb R$,
\[
  F(t)\geq
  t+\frac1{N_{\min}}\sum_{j=1}^{N_{\min}}(w_{(j)}-t)
  =A(N_{\min}).
\]
At $t=w_{(N_{\min})}$, the first $N_{\min}$ residuals are at least
$t$ and all remaining residuals are at most $t$, so equality holds.
Therefore,
$\operatorname{CVaR}_{N_{\min}/N}[g_k(z,Y)-\tau_k]=A(N_{\min})$.
Translation equivariance of CVaR yields
\(
  \max_{S\in\mathcal U}H_k(z,S)
  =\operatorname{CVaR}_{N_{\min}/N}[g_k(z,Y)]-\tau_k.
\)
The inclusion $\mathcal S(x^0)\subseteq\mathcal U$ then gives
the stated bound.
\end{proof}

The bound retains the minimum leaf size restriction but
allows subsets that need not be attainable query
neighborhoods. Its conservatism therefore depends on how
strongly the tree restrictions exclude subsets with large
average target values. For a fixed decision and training
sample, the envelope is nonincreasing in $N_{\min}$,
since it averages a larger fraction of the ordered
target values. This does not imply that the optimal
radius is nondecreasing, as changing $N_{\min}$ also changes
the nominal loss $\rho^\star$ from which the radius is
measured.

\section{Dynamic programming and constraint generation}\label{sec:dp}
 
We solve ODT-JRRO by constraint generation, with separation performed over query paths. We first show how to compute the minimum estimation loss compatible with a given query path by independently optimizing the sibling subtrees. The required subtree values are computed on demand by dynamic programming, memoized, and reused across separation calls.
The algorithm alternates between a restricted master problem and separation. The restricted master uses the entry levels of the neighborhoods discovered so far to determine a candidate decision and radius. Given this candidate, separation searches for an undiscovered neighborhood that violates a target and enters at or below the candidate radius. A newly identified neighborhood is added to the restricted master, and the process is repeated.

\subsection{Query-path decomposition}\label{sec:paths}
Since the downstream problem depends only on the leaf to which the query point is routed, it is convenient to describe a tree partition in terms of paths. A \textit{query path} is an ordered sequence $\pi=\bigl((a_1,b_1),\ldots,(a_{|\pi|},b_{|\pi|})\bigr)$ of splits with $|\pi|\leq D$, and we write $\pi_T$ for the path followed by $x^0$ in the tree $T$, that is, the sequence of splits encountered on the way from the root to $\ell_T(x^0)$. Since the branching direction at each node is determined by $x^0$ itself, the path $\pi$ determines the half-spaces
\begin{equation}
  \mathcal H_j(\pi)
  =
  \begin{cases}
    \{x\in\mathcal X: a_j^\top x \leq b_j\}, & \text{if } a_j^\top x^0 \leq b_j,\\[2pt]
    \{x\in\mathcal X: a_j^\top x > b_j\}, & \text{otherwise,}
  \end{cases}
  \qquad j=1,\ldots,|\pi|,
  \label{eq:halfspaces}
\end{equation}
and hence the query region and the sibling regions
\begin{equation}
  R(\pi)=\bigcap_{j=1}^{|\pi|}\mathcal H_j(\pi),
  \qquad
  Q_j(\pi)=\Bigl(\bigcap_{l<j}\mathcal H_l(\pi)\Bigr)\cap \mathcal H_j(\pi)^{c},
  \quad j=1,\ldots,|\pi|,
  \label{eq:box-siblings}
\end{equation}
with $R(\varnothing)=\mathcal X$. By construction $x^0\in R(\pi)$, and $\{R(\pi),Q_1(\pi),\ldots,Q_{|\pi|}(\pi)\}$ is a partition of $\mathcal X$. Figure \ref{fig:query-path} illustrates these objects for $D=2$ and $d_x=2$.
\begin{figure}[t]
\centering
\definecolor{pathblue}{RGB}{0,0,0}
\begin{tikzpicture}[
  obs/.style={circle, fill=black, inner sep=0pt, minimum size=3pt},
  qpt/.style={rectangle, draw=black, fill=white, line width=0.6pt,
              inner sep=1.7pt, rotate=45},
  lbl/.style={font=\small},
  slbl/.style={font=\footnotesize},
  scale = 0.8
]

%======================= panel (a): the partition of X =======================
\begin{scope}[scale=5]
  % regions
  \fill[gray!12]  (0,0)       rectangle (0.42,1);     % Q_1(pi)
  \fill[gray!28]  (0.42,0.63) rectangle (1,1);        % Q_2(pi)
  \fill[gray!55]  (0.42,0)    rectangle (1,0.63);     % R(pi)

  % frame and split lines
  \draw[black, line width=0.5pt] (0,0) rectangle (1,1);
  \draw[black, line width=0.9pt] (0.42,0) -- (0.42,1);
  \draw[black, line width=0.9pt] (0.42,0.63) -- (1,0.63);

  % observations: N(Q_1) = 6
  \node[obs] at (0.08,0.15) {}; \node[obs] at (0.15,0.30) {};
  \node[obs] at (0.22,0.82) {}; \node[obs] at (0.30,0.35) {};
  \node[obs] at (0.12,0.68) {}; \node[obs] at (0.36,0.10) {};
  % observations: N(Q_2) = 4
  \node[obs] at (0.55,0.78) {}; \node[obs] at (0.87,0.90) {};
  \node[obs] at (0.88,0.70) {}; \node[obs] at (0.62,0.95) {};
  % observations: N(R) = 4
  \node[obs] at (0.52,0.22) {}; \node[obs] at (0.66,0.50) {};
  \node[obs] at (0.80,0.55) {}; \node[obs] at (0.90,0.12) {};

  % query point
  \node[qpt] at (0.60,0.13) {};
  \node[slbl, anchor=west] at (0.61,0.13) {$x^0$};

  % region labels
  \node[lbl] at (0.21,0.50) {$Q_1(\pi)$};
  \node[lbl] at (0.71,0.82) {$Q_2(\pi)$};
  \node[lbl] at (0.72,0.35) {$R(\pi)$};

  % threshold ticks
  \node[slbl, anchor=north] at (0.42,-0.02) {$b_1$};
  \node[slbl, anchor=east]  at (-0.02,0.63) {$b_2$};
  \node[slbl, anchor=north] at (0.5,-0.10) {$f_1$};
  \node[slbl, anchor=east, rotate=90] at (-0.10,0.5) {$f_2$};

  %\node[lbl] at (0.5,-0.25) {(a) partition of $\mathcal X$ induced by $\pi$};
\end{scope}

%======================= panel (b): the corresponding tree ===================
\begin{scope}[xshift=11cm, yshift=4cm, scale=1,
              lv/.style={line width=1.3pt},
              qpath/.style={draw=pathblue, line width=1.7pt},
              lv2/.style={line width=0.5pt}]
              
  \coordinate (n1) at (0,0);
  \coordinate (n2) at (-2,-1.5);
  \coordinate (n3) at (2,-1.85);
  \coordinate (n4) at (0.5,-3.25);
  \coordinate (n5) at (3.5,-3.25);

  % path edges and shed edges
  \draw[lv2]   (n1) -- (n2);
  \draw[qpath] (n1) -- (n3);
  \draw[qpath] (n3) -- (n4);
  \draw[lv2]   (n3) -- (n5);

  % internal nodes on the query path
  \node[draw, circle, fill=white, inner sep=1.5pt, font=\footnotesize] at (n1)
       {$a_1^\top x\!\le\! b_1$};
  \node[draw, circle, fill=white, inner sep=1.5pt, font=\footnotesize] at (n3)
       {$a_2^\top x\!\le\! b_2$};

  % shed leaf at level 1
  \draw[lv2, fill=gray!12]
        ($(n2)+(-1.1,0)$)
        rectangle
        ($(n2)+(1.1,-0.8)$);
  \node[font=\footnotesize] at ($(-2,-1.515)+(0,-0.4)$)
        %{$\Psi_{1}\bigl(Q_1(\pi)\bigr)$};
        {$Q_1(\pi)$};

  % shed leaf at level 2: budget D-2 = 0, necessarily a single leaf
  \draw[lv2, fill=gray!28] ($(n5)+(-1.1,0)$) rectangle ($(n5)+(1.1,-0.8)$);
  \node[font=\footnotesize] at ($(3.5,-3.275)+(0,-0.4)$) %{$\Psi_{0}\bigl(Q_2(\pi)\bigr)$};
  {$Q_2(\pi)$};

  % query leaf
  \draw[lv, fill=gray!55] ($(n4)+(-1.1,0)$) rectangle ($(n4)+(1.1,-0.8)$);
  \node[font=\footnotesize] at ($(0.5,-3.275)+(0,-0.4)$) {$R(\pi)$};

  % branch annotations
  \node[slbl, anchor=south east] at ($(n1)!0.5!(n2)$) {$\le$};
  \node[slbl, anchor=south west] at ($(n1)!0.5!(2,-1.375)$) {$>$};
  \node[slbl, anchor=south east] at ($(n3)!0.5!(n4)$) {$\le$};
  \node[slbl, anchor=south west] at ($(n3)!0.5!(3.5,-3.15)$) {$>$};

  %\node[lbl] at (0.3,-3.5) {(b) a tree realizing $\pi$};
\end{scope}

\end{tikzpicture}

\caption{Query path, query region, and sibling regions for $D=2$, $d_x=2$ and
$N_{\min}=3$. The path $\pi=\bigl((a_1,b_1),(a_2,b_2)\bigr)$ retains at each
level the half-space $\mathcal H_j(\pi)$ containing $x^0$, here
$\mathcal H_1(\pi)=\{a_1^\top x > b_1\}$ and $\mathcal H_2(\pi)=\{a_2^\top x \le b_2\}$. Their intersection is the query region $R(\pi)$, whose observations form the query neighborhood
$S_T(x^0)=\mathcal N(R(\pi))$, while the complementary half-spaces yield the
sibling regions $Q_1(\pi)$ and $Q_2(\pi)$.}
\label{fig:query-path}
\end{figure}
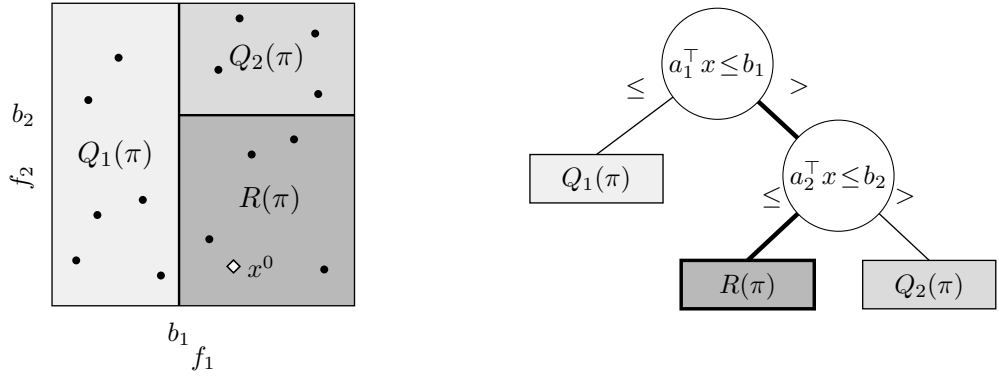
The terminal leaf of $x^0$ in $T$ is the intersection of the half-spaces
selected along its root-to-leaf path, that is, $R(\pi_T)$, and an observation
reaches that leaf if and only if it satisfies the same branching conditions. The
query neighborhood \eqref{eq:query-leaf-index-set} is therefore
\mbox{$S_T(x^0)=\mathcal N\bigl(R(\pi_T)\bigr)$}, so that two trees sharing
the same query path agree on the query leaf and the observations it contains, and differ
only in the way they partition the sibling regions. The thresholds $b_j$ in \eqref{eq:halfspaces} may be restricted to a finite
menu without changing any attainable loss or query neighborhood. Indeed, for
each feature $f\in[d_x]$ it is sufficient to place one threshold between
consecutive distinct values of
\begin{equation}
  V_f=\{x_{if}:i\in[N]\}\cup\{x^0_f\},
  \label{eq:menu}
\end{equation}
since two thresholds lying between the same pair of consecutive values
induce the same partition of $\{ x_i : i \in [N]\} \cup\{x^0\}$. We denote by
$M_{\mathrm{thr}}\leq Nd_x$ the total number of candidate thresholds so
obtained. A query region is an intersection of at most $D$ of the
corresponding half-spaces, so that
\begin{equation}
  |\mathcal S(x^0)|\ \leq\ \sum_{s=0}^{D}\binom{M_{\mathrm{thr}}}{s}\ =\ O\bigl((Nd_x)^{D}\bigr).
  \label{eq:counting}
\end{equation}
This gives a bound that is polynomial in $N$ for fixed depth. We denote by
$\Pi_D(x^0)$ the set of feasible paths of maximum depth $D$ and minimum leaf
size $N_{\min}$ built on this menu, so that the set of attainable
neighborhoods is exactly
\mbox{$\{ \mathcal N(R(\pi)) : \pi\in\Pi_D(x^0)\}$}.
Two trees that share the same query path may still differ within the sibling
regions, and the minimum loss \eqref{eq:neighborhood-loss} seeks the best
such tree. For an index set $A\subseteq[N]$ and an integer depth budget $h\geq0$, let $\Psi_h(A)$ denote the least contribution to the loss \eqref{eq:tree-loss} of a feasible subtree of depth at most $h$ receiving the observations indexed by $A$, that is
\begin{equation}
\Psi_h(A)
=
\min_{T'}
\left\{
    \frac{1}{N}
    \sum_{\ell\in\mathcal L(T')}
    \operatorname{SSE}\bigl(\mathcal N(\ell)\bigr)
    +\lambda|\mathcal I(T')|
\right\},
\label{eq:psi}
\end{equation}
with the convention $\min\varnothing:=+\infty$. The minimum is taken over axis-aligned subtrees using the split menu \eqref{eq:menu}, whose leaves partition $A$, receive at least $N_{\min}$ observations each, and have depth at most $h$. Thus, $\Psi_h(A)=+\infty$ whenever $|A|<N_{\min}$. If $|A|\geq N_{\min}$, the single-leaf subtree is feasible for every $h\geq0$, so $\Psi_h(A)$ is finite. We now show how the optimal loss can be decomposed. Once the query path is fixed, the sibling-subtree optimization problems are separable. The maximum depth and minimum leaf size restrictions can be enforced separately within each subtree, and the loss is additive across leaves and internal splits.

\begin{proposition}\label{prop:decomposition}
For the tree class $\Theta_N$ defined above, with maximum depth $D$, minimum leaf size $N_{\min}$, axis-aligned splits from \eqref{eq:menu}, and the additive loss \eqref{eq:tree-loss}, the least regularized loss attainable by a tree that routes $x^0$ along a path $\pi\in\Pi_D(x^0)$ is
\begin{equation}
  \rho_\pi(x^0)
  :=
  \min_{T\in\Theta_N:\ \pi_T=\pi} \rho(T;\mathcal D_N)
  =
  \underbrace{
    \vphantom{\displaystyle
      \sum_{j=1}^{|\pi|}
      \Psi_{D-j}\bigl(\mathcal N(Q_j(\pi))\bigr)}
    \frac{1}{N}\operatorname{SSE}\bigl(\mathcal N(R(\pi))\bigr)
  +
    \vphantom{\displaystyle
      \sum_{j=1}^{|\pi|}
      \Psi_{D-j}\bigl(\mathcal N(Q_j(\pi))\bigr)}
    \lambda|\pi|
  }_{\text{loss at query leaf}}
  +
  \underbrace{
    \sum_{j=1}^{|\pi|}
      \Psi_{D-j}\bigl(\mathcal N(Q_j(\pi))\bigr)
  }_{\text{loss at other leaves}},
  \label{eq:route-loss}
\end{equation}
and $\rho_S(x^0)=\min\{\rho_\pi(x^0):\pi\in\Pi_D(x^0),\ \mathcal N(R(\pi))=S\}, \, \forall S\in\mathcal S(x^0)$.
\end{proposition}

\begin{proof}
Fix a feasible path $\pi\in\Pi_D(x^0)$ and let $T\in\Theta_N$
satisfy $\pi_T=\pi$. By the argument following \eqref{eq:menu},
we may restrict the thresholds in the sibling subtrees to the
menu while leaving $\pi$ unchanged, since this preserves the
training-data partition and the loss.
The query leaf of $T$ is $R(\pi)$, and the complement
$\mathcal X\setminus R(\pi)$ is partitioned through the subtrees
rooted in the sibling regions $Q_j(\pi)$,
$j=1,\ldots,|\pi|$. The $j$-th subtree is rooted at depth $j$
and therefore has depth budget $D-j$. These subtrees act on
pairwise disjoint index sets and are subject to no joint
constraint.
By \eqref{eq:tree-loss}, the loss decomposes into the query leaf
term, the penalty $\lambda|\pi|$ for the path splits, and
the sibling-subtree contributions. Each such contribution
consists of the squared errors of its leaves, normalized by
$N$, and the penalty for its own splits, and is therefore
at least $\Psi_{D-j}\bigl(\mathcal N(Q_j(\pi))\bigr)$.
Hence, every tree with query path $\pi$ has loss at least
the right-hand side of \eqref{eq:route-loss}.
Conversely, feasibility of $\pi$ ensures that the query leaf
and every sibling region contain at least $N_{\min}$
observations. Each sibling-subtree problem is therefore
feasible, since a single leaf is admissible, and its minimum
is attained because the split menu and depth bound permit
only finitely many tree structures. Attaching an optimal
subtree to each sibling region yields a tree of depth
at most $D$ satisfying the minimum leaf size requirement.
The query path remains $\pi$, since each attached subtree
lies on the side of its path split that does not contain
$x^0$ and is therefore reached only by points leaving the
path. The resulting tree attains the preceding lower bound,
proving \eqref{eq:route-loss}.
Finally, the menu reduction preserves both the loss and
the query neighborhood, so it suffices to consider trees
whose query paths belong to $\Pi_D(x^0)$. Since
$S_T(x^0)=\mathcal N(R(\pi_T))$, minimizing first over trees
sharing a query path and then over paths $\pi$ satisfying
$\mathcal N(R(\pi))=S$ gives the second claim.
\end{proof}

Proposition \ref{prop:decomposition} allows us to price a query path without
ever constructing the tree that realizes it, since the completion of a fixed
query path decomposes into at most $D$ independent subtree problems on disjoint
data. We remark that two paths that use the same splits in different orders
yield the same query region $R(\pi)$, but different sibling regions with
different depth budgets. The quantity $\rho_S(x^0)$ is therefore a minimum
over such orders.

\subsection{Bellman recursion}\label{sec:dp-values}
 
We compute the subtree values \eqref{eq:psi} by dynamic programming over index sets. Fix $A\subseteq[N]$ and an integer depth budget $h\geq0$. If $|A|<N_{\min}$, the state is infeasible and has value $+\infty$. Otherwise, a subtree receiving the observations indexed by $A$ may consist of a single leaf. When $h\geq1$, it may instead split $A$ using the menu \eqref{eq:menu}, provided that both children contain at least $N_{\min}$ observations. The child subtrees can then be optimized independently with depth budget $h-1$. Minimizing over these alternatives gives the Bellman recursion
\begin{equation}
\Psi_h(A) =
\begin{cases}
    +\infty, & |A| < N_{\min}, \\[2pt]
    \displaystyle\frac{1}{N}\operatorname{SSE}(A), & |A|\geq N_{\min},\ h = 0, \\[6pt]
    \displaystyle
    \min\Biggl\{
        \frac{1}{N}\operatorname{SSE}(A),\;
        \min_{(f,b)}
        \Bigl[
            \lambda
            + \Psi_{h-1}(A^-_{f,b})
            + \Psi_{h-1}(A^+_{f,b})
        \Bigr]
    \Biggr\},
    & \text{otherwise,}
\end{cases}
\label{eq:dp}
\end{equation}
where $A^-_{f,b} = \{i\in A:x_{if}\le b\}$ and
$A^+_{f,b} = A\setminus A^-_{f,b}$. The inner minimum runs over the
menu splits for which both children retain at least $N_{\min}$
observations and is $+\infty$ when none exists. All leaf contributions
are normalized by the global sample size $N$.
Values and minimizers are memoized on $(A,h)$ and reused across subsequent
calls for the same instance. In particular,
$\rho^\star=\Psi_D([N])$, and the stored minimizers reconstruct
the nominal tree $T^0$ of \eqref{eq:lexicographic}. This also gives $S_1=S_{T^0}(x^0)$ and
$e_{S_1}(x^0)=0$.
 
\subsection{Minimum-loss separation}\label{sec:dp-sep}
 
Let $\mathcal C\subseteq\mathcal S(x^0)$ denote the set of
discovered neighborhoods. Given an incumbent $(\hat z,\hat r)$
and a target $k$, define the observation-wise residuals
\(
    w_{ik}:=g_k(\hat z,y_i)-\tau_k,
    \,i\in[N].
\)
The separation problem seeks an undiscovered neighborhood
of minimum compatible tree loss that violates target $k$
and enters at or below $\hat r$, that is
\begin{subequations}\label{eq:sep}
\begin{align}
  (\operatorname{SEP}_k(\hat z,\hat r;\mathcal C))\qquad
  \min_{\pi}\quad & \rho_\pi(x^0) \\
  \textnormal{s.t.}\quad
  & \sum_{i\in\mathcal N(R(\pi))} w_{ik}>0,
  \label{eq:sep-viol}\\
  & \rho_\pi(x^0) \leq \rho^\star+\hat r, \label{eq:sep-cap}\\
  & \mathcal N(R(\pi))\notin\mathcal C, \label{eq:sep-excl}\\
  & \pi\in\Pi_D(x^0).
\end{align}
\end{subequations}
If $\max_{i\in[N]}w_{ik}\leq0$, then no neighborhood violates
target $k$ and \eqref{eq:sep} is infeasible.
We solve \eqref{eq:sep} by depth-first search over partial paths. A partial
path is represented by the index set $A$ of the observations still routed
with $x^0$, the remaining depth $h$, and the cost $c$ already committed by
the splits taken so far and their sibling subtrees. The search starts from
$A=[N]$, $h=D$ and $c=0$. An admissible split retains the child $A'$
containing $x^0$ and forms the sibling set $B=A\setminus A'$, with both
children containing at least $N_{\min}$ observations. It moves to $A'$
with remaining depth $h-1$ and cost $c+\lambda+\Psi_{h-1}(B)$, which adds
to $c$ the penalty of the split and the optimal completion of the sibling
set. Stopping at the current node instead makes $A$ the query leaf and
completes the path at price
$c+\frac{1}{N}\operatorname{SSE}(A)$, which is \eqref{eq:route-loss}.
The search is based on the following pruning rules:
\begin{itemize}
    \item Every completion of a partial path extends the current node into
    a subtree of depth at most $h$ receiving $A$, so by
    \eqref{eq:route-loss} its price is $c$ plus the contribution of that
    subtree, and therefore at least $c+\Psi_h(A)$. This bound prunes any
    partial path for which it exceeds the loss cap $\rho^\star+\hat r$ or
    the price of the best feasible candidate found so far.
    \item Within a separation call, the completions available at a node
    depend only on $A$ and $h$. Each such completion adds the same amount
    to the price regardless of the path used to reach that state.
    A partial path that reaches the same $A$ and $h$ with an equal or greater committed cost is therefore dominated, and the search keeps only the
    cheapest cost for each pair $(A,h)$.
\end{itemize}
A completed path is a candidate for \eqref{eq:sep} when its neighborhood
violates the target, has not been discovered before, and respects the loss
cap. The finite search terminates with a global minimizer of
\eqref{eq:sep} or with a proof that no feasible path exists.
If $\hat\pi$ is a global minimizer and
$\widehat S=\mathcal N(R(\hat\pi))$, then
$\rho_{\hat\pi}(x^0)=\rho_{\widehat S}(x^0)$. Indeed, a cheaper path
inducing the same neighborhood would preserve both the target violation
and the exclusion condition, would respect the loss cap, and would
contradict optimality of $\hat\pi$. Consequently, separation returns the
exact entry level $e_{\widehat S}(x^0)=\rho_{\hat\pi}(x^0)-\rho^\star$. %Constraints \eqref{eq:sep-cap} and \eqref{eq:sep-viol} are imposed up to floating-point precision. 

\subsection{Constraint generation algorithm}\label{sec:dp-cg}
 
We initialize $\mathcal C=\{S_1\}$ and alternate two steps. The
restricted master applies the monotone scan of Proposition
\ref{prop:closed-form} to the entry levels of $\mathcal C$, adding
all cuts at a common level together. The master returns the decision
$\hat z$ minimizing $H_0(z,S_1)$ on the last nonempty restricted feasible
set, using the tie-break of \eqref{eq:lexicographic}.
When the scan first becomes infeasible at a discovered level,~$\hat r$ is
that level and is an excluded breakdown endpoint. Otherwise,
$\hat r=\bar r$ and the candidate decision satisfies every discovered
constraint at the cap. In both cases, $\hat r$ is an upper bound on the
optimal radius of the full problem.
Separation is then performed for each target until a new violating
neighborhood is found, which is added to $\mathcal C$. Since separation imposes
$\rho_\pi(x^0)\le\rho^\star+\hat r$ and $\hat r\le\bar r$, every added
neighborhood enters at or below the cap. The loss cap in separation is
inclusive, as violations at $\bar r$ must be ruled out before
declaring that the cap is attained.
Algorithm \ref{alg:dp} summarizes the procedure. 
It performs at most $m-1$ additions to $\mathcal C$,
at most $m$ restricted-master scans and at most $(K+1)m$ separation calls, where $m=|\mathcal S(x^0)|$ so that
$m=O\bigl((Nd_x)^D\bigr)$ is polynomial for fixed $D$ by \eqref{eq:counting}.
 
\begin{algorithm}[htb]
\caption{Constraint generation for ODT-JRRO}\label{alg:dp}
\footnotesize
\begin{algorithmic}[1]
\Require $\mathcal D_N$, $x^0$, $\Theta_N$, $\mathcal Z$,
         targets $(g_k,\tau_k)_{k=0}^{K}$, cap $\bar r$ from \eqref{eq:radius-cap}
\Ensure $(z^\star,r^\star,\texttt{attained})$ or \textsc{Infeasible}
\State $\rho^\star\gets\Psi_D([N])$, fix $T^0$ and $S_1 = S_{T^0}(x^0)$
\State $\mathcal C\gets\{S_1\}$
\State $e_{S_1}(x^0) \gets 0$
\Loop
  \State let $0=\hat\eta_1<\cdots<\hat\eta_P\leq\bar r$ be the entry levels
         present in $\mathcal C$, and let $\widehat{\mathcal F}_p$ be the set
         \eqref{eq:prefix} restricted to the neighborhoods of $\mathcal C$
         with entry level at most $\hat\eta_p$
  \State $p^\star\gets\max\bigl\{p\in[P]:\widehat{\mathcal F}_p\neq\varnothing\bigr\}$,
         and $p^\star\gets0$ if no such $p$ exists
  \If{$p^\star=0$}
    \State \Return \textsc{Infeasible}
  \EndIf
  \State choose $\hat z\in\arg\min_{z\in\widehat{\mathcal F}_{p^\star}}
         H_0(z,S_1)$
  \If{$p^\star<P$}
    \State $\hat r\gets\hat\eta_{p^\star+1}$ and
           $\texttt{attained}\gets\texttt{false}$
  \Else
    \State $\hat r\gets\bar r$ and
           $\texttt{attained}\gets\texttt{true}$
  \EndIf
  \State $\texttt{found}\gets\texttt{false}$
  \For{$k=0,\ldots,K$}
    \State $w_{ik}\gets g_k(\hat z,y_i)-\tau_k$ for every $i\in[N]$
    \If{$\max_{i\in[N]}w_{ik}\leq0$}
      \State \textbf{continue}
    \EndIf
    \State solve $\operatorname{SEP}_k(\hat z,\hat r;\mathcal C)$
    \If{an optimal path $\hat\pi$ is returned}
      \State $\widehat S\gets\mathcal N(R(\hat\pi))$
      \State $e_{\widehat S}(x^0)\gets\rho_{\hat\pi}(x^0)-\rho^\star$
        \State $\mathcal C\gets\mathcal C\cup\{\widehat S\}$
      \State $\texttt{found}\gets\texttt{true}$
      \State \textbf{break}
    \EndIf
  \EndFor
  \If{\textbf{not} $\texttt{found}$}
    \State $(z^\star,r^\star)\gets(\hat z,\hat r)$
  \State \Return $(z^\star,r^\star,\texttt{attained})$
  \EndIf
\EndLoop
\end{algorithmic}
\end{algorithm}

\section{Computational evaluation}\label{sec:computational}
In this section, we conduct a computational evaluation of the proposed method for solving ODT-JRRO. We use the contextual newsvendor problem as the downstream problem. The decision is an order quantity $q\in\mathcal Z=[0,\bar q]$,
the uncertain outcome $Y$ is demand, and the cost of ordering $q$ against a realized demand $y$ is
\begin{equation}
    c(q,y)
    =c_{\mathrm{u}}(y-q)^+
     +c_{\mathrm{o}}(q-y)^+,
    \label{eq:nv-cost}
\end{equation}
with underage cost $c_{\mathrm{u}}>0$ and overage cost $c_{\mathrm{o}}>0$. There is a
single target, obtained by setting $g_0\equiv c$ and $\tau_0\equiv\tau$.
By \eqref{eq:local-empirical-target}, a tree $T$ induces at the query
point the average cost of $q$ over the historical demands sharing the
query leaf, and the decision is acceptable for $T$ when
$G_0(q,T;x^0)\le\tau$. When feasible, ODT-JRRO returns a decision $q^\star$ and the supremum
$r^\star$ of the feasible radii. We take $c_{\mathrm{u}}=3$, $c_{\mathrm{o}}=1$ and
$\bar q=\max\{0,\max_i y_i+8\}$, unless otherwise specified.
The function $c(\cdot,y)$ is convex and piecewise linear, so each robust
cut \eqref{eq:tree-cut} describes a possibly empty interval of
$[0,\bar q]$, and every feasible set \eqref{eq:prefix} is an
intersection of such intervals. The monotone scan of
Proposition~\ref{prop:closed-form} can therefore be carried out by
updating interval endpoints, without solving an auxiliary optimization
problem. We use the tie-breaking rule \eqref{eq:lexicographic} and the radius cap \eqref{eq:radius-cap}. ODT-JRRO is solved using absolute and relative tolerances of $10^{-10}$ each. All experiments were performed on a workstation equipped with an Intel Core Ultra 7 265 processor and 32 GB of RAM, running Windows 11 Enterprise. The implementation uses Python 3.13, NumPy 2.4.6, and Numba 0.65.1. 

\subsection{Evaluation on synthetic data}\label{sec:e1}
We evaluate the computational performance of Algorithm~\ref{alg:dp} on
synthetic instances. 
We generate Gaussian demands with
$Y\mid X=x\sim\mathcal N(\mu(x),\sigma^2)$ and
$\sigma\in\{1,2,5\}$, so that the population cost of any order quantity is available in closed form. Writing
$z'=(q-\mu(x))/\sigma$ and letting $\phi,\Phi$ denote the standard normal
density and distribution function, respectively, we have
\begin{equation}
  \mathbb E\bigl[c(q,Y)\mid X=x\bigr]
  \;=\;\sigma\Bigl\{c_{\mathrm{o}}\bigl(\phi(z')+z'\,\Phi(z')\bigr)
      +c_{\mathrm{u}}\bigl(\phi(z')-z'\,(1-\Phi(z'))\bigr)\Bigr\}.
  \label{eq:nv-true-cost}
\end{equation}
The unconstrained minimizer is the conditional quantile
$q^{\mathrm{orc}}(x)=\mu(x)+\sigma\,
\Phi^{-1}\!\bigl(c_{\mathrm{u}}/(c_{\mathrm{o}}+c_{\mathrm{u}})\bigr)$, at which
\eqref{eq:nv-true-cost} reduces to
\begin{equation}
  c^{\mathrm{orc}}(\sigma)
  =\sigma(c_{\mathrm{o}}+c_{\mathrm{u}})\,\phi\!\left(\Phi^{-1}\!\left(\frac{c_{\mathrm{u}}}{c_{\mathrm{o}}+c_{\mathrm{u}}}\right)\right).
  \label{eq:nv-oracle-cost}
\end{equation}
Under homoscedastic noise, this value does not depend on $x$. With the
chosen costs, $c^{\mathrm{orc}}(\sigma)\simeq1.2711\sigma$.
If the oracle order
falls outside $[0,\bar q]$, the constrained minimizer is its projection
onto that interval.
Covariates are drawn uniformly on $[0,1]^{d_x}$ and rounded to four
decimal places. For $d_x\in\{2,4\}$, mean demand takes values $12,18,25,33$
on the lower-left, upper-left, lower-right and upper-right quadrants
of the first two coordinates, split at $\tfrac12$. The additional
covariates are irrelevant but available for splitting. For $d_x=1$,
mean demand is $12$ for $x_1\le\tfrac12$ and $25$ otherwise.
These mean functions can be represented by trees of depth at most two before the sample-dependent split menu and minimum leaf size are imposed.
To avoid the boundary of the covariate support, we draw the query point uniformly from $[0.2,0.8]^{d_x}$ and round its coordinates to four decimal places.
We set the target
$\tau=(1+\gamma)c^{\mathrm{orc}}(\sigma)$, where $c^{\mathrm{orc}}(\sigma)$ is
the optimal cost under the true data-generating model and $\gamma\geq0$ is the
relative slack allowed above it. The oracle cost is computed directly from the
model, so the target is exogenous to the data used to fit the trees.
We perform ten replications for each combination of sample size
$N\in\{50,100,150,200\}$, noise level $\sigma\in\{1,2,5\}$,
target slack $\gamma\in\{0.10,0.25,0.50,0.75,1\}$, maximum depth
$D\in\{2,3,4\}$, minimum leaf size $N_{\min}\in\{1,2,5,10\}$,
split penalty $\lambda\in\{0,0.1,0.25,0.5,0.75,1\}$, and number
of features $d_x\in\{1,2,4\}$, for a total of $129600$ runs.
Each run is subject to a limit of $10^6$ dynamic programming cache
entries, a wall-clock limit of $12000$ seconds, and a limit of $2000$
constraint generation iterations.
Table~\ref{tab:comp-scaling} reports the percentage of runs solved and the median
and 90th percentile wall-clock times by $D$, $N$ and $d_x$. Table~\ref{tab:comp-profile} reports the computational profile by depth and target slack, together with radius statistics computed over solved runs only.

\begin{table}[htb]
\centering
\footnotesize
\setlength{\tabcolsep}{1pt}
\caption{Computational performance of Algorithm~1 by maximum depth $D$, sample size $N$ and number of features $d_x$. Each cell aggregates 3600 runs. \emph{Solved} is the percentage of runs that returned a proven outcome up to numerical tolerance (certified radius, attained cap, or infeasibility at $r=0$). All remaining runs stopped at the limit of $10^6$ cache entries. We report the median and 90th percentile wall-clock times in seconds over solved runs.}
\label{tab:comp-scaling}
\begin{tabular*}{\textwidth}{@{\extracolsep{\fill}}crrrrrrrrrr}
\toprule
 & & \multicolumn{3}{c}{$d_x=1$} & \multicolumn{3}{c}{$d_x=2$} & \multicolumn{3}{c}{$d_x=4$} \\
\cmidrule(lr){3-5} \cmidrule(lr){6-8} \cmidrule(lr){9-11}
$D$ & $N$ & Solved & Median (s) & 90th (s) & Solved & Median (s) & 90th (s) & Solved & Median (s) & 90th (s) \\
\midrule
\multirow{4}{*}{2} & 50 & 100 & $<$0.01 & 0.02 & 100 & $<$0.01 & 0.03 & 100 & 0.02 & 0.07 \\
 & 100 & 100 & 0.08 & 0.18 & 100 & 0.02 & 0.23 & 100 & 0.05 & 0.51 \\
 & 150 & 100 & 0.19 & 0.49 & 100 & 0.04 & 0.57 & 100 & 0.08 & 1.21 \\
 & 200 & 100 & 0.38 & 1.33 & 100 & 0.06 & 1.26 & 100 & 0.12 & 2.62 \\
\midrule
\multirow{4}{*}{3} & 50 & 100 & 0.04 & 0.12 & 100 & 0.10 & 0.20 & 100 & 0.42 & 0.78 \\
 & 100 & 100 & 0.61 & 2.30 & 100 & 0.59 & 1.45 & 100 & 2.73 & 4.86 \\
 & 150 & 100 & 2.14 & 7.71 & 100 & 1.57 & 4.23 & 95.0 & 7.00 & 10.4 \\
 & 200 & 100 & 5.28 & 23.0 & 99.1 & 3.25 & 8.35 & 80.1 & 13.2 & 16.2 \\
\midrule
\multirow{4}{*}{4} & 50 & 100 & 0.08 & 0.33 & 100 & 0.83 & 1.29 & 100 & 7.94 & 13.2 \\
 & 100 & 100 & 1.52 & 6.01 & 100 & 10.1 & 13.7 & 0.4 & 13.8 & 16.8\\
 & 150 & 100 & 6.21 & 24.6 & 23.3 & 20.6 & 26.3 & 0 & -- & -- \\
 & 200 & 100 & 15.6 & 69.5 & 0 & -- & -- & 0 & -- & -- \\
\bottomrule
\end{tabular*}
\end{table}

\begin{table}[htb]
\centering
\footnotesize
\setlength{\tabcolsep}{1pt}
\caption{Computational profile of Algorithm~\ref{alg:dp}
by maximum depth $D$ and target slack $\gamma$, with
8640 runs per row. \emph{Solved} is the percentage of runs
returning a proven outcome up to numerical tolerance. \emph{Stop in $\rho^\star$}
is the percentage of unsolved runs that reached the cache
limit while computing the nominal optimum $\rho^\star$.
\emph{Outcome} splits solved runs into infeasibility at
$r=0$ (Inf.), an excluded breakdown threshold (Brk.),
and an attained cap $\bar r$ (Cap). The median normalized
value $r^\star/\rho^\star$ and the percentage of decisions
differing from the nominal decision $z^0$ are computed over solved Brk.\ and
Cap outcomes only. \emph{Separations} reports the median
and maximum number of calls to the separation problem,
and \emph{States} reports the median number of memoized
subtree values $\Psi_h(A)$, in thousands. Both are computed
over all solved runs, including infeasible outcomes.
\emph{Time share} splits the solved runtime into evaluation
of the dynamic program (DP), path search excluding DP
(Path), and the master problem (Master). The remainder is fixed overhead such as model setup and result assembly.}
\label{tab:comp-profile}
\begin{tabular*}{\textwidth}{@{\extracolsep{\fill}}ccrrrrrrrrrrrrr}
\toprule
 & & \multicolumn{2}{c}{Runs (\%)} & \multicolumn{3}{c}{Outcome (\%)} & Median & $z^\star\neq z^0$ & \multicolumn{2}{c}{Separations} & States & \multicolumn{3}{c}{Time share (\%)} \\
\cmidrule(lr){3-4} \cmidrule(lr){5-7} \cmidrule(lr){10-11} \cmidrule(lr){13-15}
$D$ & $\gamma$ & Solved & Stop in $\rho^\star$ & Inf. & Brk. & Cap & $r^\star/\rho^\star$ & (\%) & Med. & Max & ($\times 10^3$) & DP & Path & Master \\
\midrule
\multirow{5}{*}{2} & 0.10 & 100 & -- & 20.4 & 79.6 & 0 & 0.17 & 22.3 & 2 & 9 & 1.6 & 67.1 & 29.0 & 2.0 \\
 & 0.25 & 100 & -- & 10.7 & 89.3 & 0 & 0.23 & 32.0 & 2 & 13 & 2.1 & 59.8 & 37.3 & 1.4 \\
 & 0.50 & 100 & -- & 9.2 & 90.7 & $<$0.1 & 0.36 & 35.4 & 2 & 21 & 2.8 & 55.4 & 43.3 & 0.7 \\
 & 0.75 & 100 & -- & 8.8 & 89.2 & 2.0 & 0.50 & 42.2 & 2 & 33 & 3.2 & 52.9 & 45.6 & 0.7 \\
 & 1.00 & 100 & -- & 8.1 & 85.1 & 6.8 & 0.65 & 40.7 & 2 & 49 & 3.7 & 53.8 & 44.8 & 0.6 \\
\midrule
\multirow{5}{*}{3} & 0.10 & 100 & -- & 18.8 & 81.2 & 0 & 0.09 & 29.7 & 2 & 17 & 21 & 93.4 & 6.5 & $<$0.1 \\
 & 0.25 & 99.5 & 0 & 11.1 & 88.9 & 0 & 0.12 & 43.5 & 2 & 15 & 23 & 85.4 & 14.0 & $<$0.1 \\
 & 0.50 & 98.3 & 0 & 10.2 & 89.7 & $<$0.1 & 0.18 & 43.5 & 2 & 23 & 26 & 74.8 & 24.8 & $<$0.1 \\
 & 0.75 & 96.5 & 0 & 9.4 & 88.5 & 2.1 & 0.26 & 46.6 & 2 & 39 & 27 & 63.9 & 35.7 & $<$0.1 \\
 & 1.00 & 94.8 & 0 & 8.9 & 84.2 & 7.0 & 0.33 & 44.6 & 2 & 66 & 26 & 62.3 & 37.3 & $<$0.1 \\
\midrule
\multirow{5}{*}{4} & 0.10 & 60.6 & 98.3 & 21.3 & 78.7 & 0 & 0.14 & 33.5 & 2 & 15 & 26 & 88.8 & 11.2 & $<$0.1 \\
 & 0.25 & 60.4 & 97.8 & 14.6 & 85.4 & 0 & 0.20 & 43.2 & 2 & 17 & 27 & 78.1 & 21.7 & $<$0.1 \\
 & 0.50 & 60.3 & 97.6 & 12.7 & 87.2 & 0.1 & 0.33 & 41.9 & 2 & 26 & 27 & 65.5 & 34.4 & $<$0.1 \\
 & 0.75 & 60.2 & 97.2 & 11.1 & 85.5 & 3.3 & 0.55 & 49.7 & 2 & 44 & 30 & 54.7 & 45.2 & $<$0.1 \\
 & 1.00 & 60.1 & 97.0 & 10.2 & 78.8 & 11.0 & 0.82 & 47.7 & 2 & 69 & 31 & 55.1 & 44.8 & $<$0.1 \\
\bottomrule
\end{tabular*}
\end{table}

The results show that relaxing the target does not necessarily make
ODT-JRRO easier to solve. Indeed, achieving a larger radius requires
certifying robustness over a wider family of near-optimal trees,
increasing the burden on separation. However, at larger depths,
computing the nominal optimum becomes the dominant obstacle, so
improvements to constraint generation alone would have limited impact on
scalability. The small median number of separation calls indicates that few new neighborhoods
are typically needed before termination.
Together with the runtime decomposition, this suggests that most of the computational effort is spent computing subtree values and pricing query paths. The share of runtime devoted to path search generally increases with the target slack $\gamma$, consistent with the greater separation effort needed to certify larger robustness radii. Moreover, joint radius optimization frequently changes the selected decision relative to the nominal prescription. The frequency of these changes is generally higher at larger target slack values.

\begin{table}[htb]
\centering
\footnotesize
\setlength{\tabcolsep}{6pt}
\renewcommand{\arraystretch}{1.15}
\caption{Computational comparison with complete-tree enumeration,
full query-neighborhood enumeration, and the pruning ablation.
Entries are means over five training seeds and ten
queries per seed, with depth $D=2$, $d_x=4$, $N_{\min}=5$, and
$\lambda=0.1$, on instances generated with $\sigma = 2$ and solved with $\gamma = 1$. The upper block reports cached
dynamic programming (DP) states, while the tree enumeration reports distinct
recursive enumeration subproblems.
All reported runs were solved, and all paired comparisons with Algorithm \ref{alg:dp} agreed
with absolute and relative comparison tolerances of $10^{-7}$.}
\label{tab:enumeration-comparison}
\begin{tabular}{@{}rlrrr@{}}
\toprule
$N$ & Method & \shortstack{Total\\(s)} & \shortstack{Generated\\neighborhoods}
    & \shortstack{Distinct DP\\subtree states} \\
\midrule
50  & Constraint generation (Alg. \ref{alg:dp})            & 0.0130 &    2.88 &     925 \\
    & Constraint generation without pruning  & 0.0841 &    2.88 &   7819 \\
    & Full neighborhood enumeration & 0.0963 &   6215 &   6446 \\
\midrule
100 & Constraint generation (Alg. \ref{alg:dp})                    & 0.0359 &    2.96 &   2962 \\
    & Constraint generation without pruning  & 0.5674 &    2.96 &  42482 \\
    & Full neighborhood enumeration & 0.7784 &  31573 &  30939 \\
\midrule
200 & Constraint generation (Alg. \ref{alg:dp})                 & 0.1204 &    3.00 &  10453 \\
    & Constraint generation without pruning  & 2.5149 &    3.00 & 192133 \\
    & Full neighborhood enumeration & 4.6969 & 141351 & 130897 \\
\midrule
\addlinespace[0.5em]
 &  &  & \shortstack{Generated\\trees}
    & \shortstack{Distinct recursive\\subtree states} \\
\midrule
50 & Tree enumeration & 18.7324 & 1235809 & 12818 \\
\bottomrule
\end{tabular}
\end{table}

To assess computational efficiency, we compare Algorithm \ref{alg:dp} with full neighborhood enumeration and, for the smallest sample size, explicit tree enumeration, using matched training samples and queries. An additional ablation disables pruning to isolate its contribution. All methods are charged for the common nominal dynamic programming preprocessing, and each query is solved without reusing state from other queries. Table~\ref{tab:enumeration-comparison} shows agreement with the enumeration references within numerical tolerances. The runtime advantage of Algorithm \ref{alg:dp} over full neighborhood enumeration increases with sample size, while its generated neighborhood set remains small. The pruning ablation requires substantially more subtree states despite generating the same neighborhoods, indicating that completion-cost bounds reduce the effort needed to discover and certify them. Finally, the tree comparison illustrates how many distinct trees induce the same query neighborhood, supporting the neighborhood representation and its combination with selective constraint generation.

\subsection{Comparison against other methods on real-world data}
We benchmark ODT-JRRO on chicken demand from the YAZ dataset
provided by the Python package \textit{ddop}
\citep{philippi2021ddop}. It contains 765 records of demand for
seven ingredients at a restaurant in Stuttgart.
We encode weekday and month using one-hot indicators and retain
year, holiday and closure indicators, and the five weather
variables, giving 27 features. Date and weekend are excluded.
The weather covariates are realized same-day values recorded in the dataset, which we treat as known when constructing the contextual orders. The experiment therefore compares prescriptions conditional on observed weather.
We split the observations chronologically into 459 training,
153 calibration, and 153 test observations. The test dates run
from 8 June to 7 November 2015, and the training sample remains
fixed throughout the experiment. Demand is normalized
by its training mean, and the order capacity is twice the
maximum training demand, giving $\bar q=156$ in original units.
With $\gamma=0$, the target $\tau\approx18.0719$ is the mean
calibration cost of the order obtained by sample average approximation over the full training sample.

We compare our approach with five alternatives.
Two baselines use the partition returned by nominal tree
optimization. ODT-SAA minimizes the empirical average cost
over the observations in the query leaf, whereas ODT-mean
replaces the within-leaf empirical distribution by a point
mass at the leaf mean. This comparison holds the nominal
partition fixed and isolates the decision rule.
We select the maximum depth, minimum leaf size and split penalty from the grid
$D \in \{2,3\}$, $N_{\min} \in \{1,2,5,10,15,20\}$ and
$\lambda \in \{0, 10^{-4}, 10^{-3}, 10^{-2}, 10^{-1}, 1\}$, fitting on the training sample and
choosing the combination that minimizes the mean calibration cost of the ODT-SAA orders,
without using the test period. This gives $D=3$, $N_{\min}=20$ and $\lambda=0.01$, which are
shared by all tree-based methods.
The k-nearest neighbors (kNN-SAA) method selects its neighborhood size by minimizing
mean calibration cost, giving $k=5$. JERO-OLS applies JERO
to squared-error loss sublevel sets of a linear predictor fitted by ordinary least squares (OLS).
Both methods use features standardized with training-sample
statistics. For JERO-OLS, Sunday and December serve as reference levels, so that the design matrix has full column rank. Global-SAA minimizes empirical average cost
over the entire training sample. We use $\mathrm{CVaR}_{0.1}$ to denote the empirical mean of the
worst 10\% of realized costs, including fractional probability mass
at the tail boundary. We call a day a stockout day if the order is below realized demand.

\begin{table}[htb]
\centering
\caption{Held-out comparison on chicken demand from the YAZ dataset over the
153 consecutive days of the held-out period, with $D=3$, $N_{\min}=20$,
$\lambda=0.01$ and $\gamma=0$ ($\tau\approx18.07$).
Differences are computed as method minus ODT-JRRO, so positive values
favor ODT-JRRO. The pointwise $95\%$ intervals are percentile intervals from a paired
moving-block bootstrap with blocks of 7 consecutive days and
50000 resamples. Operational success is the percentage of days
on which a method issues an order with realized cost at most $\tau$.
All methods issue orders on every test date.}
\label{tab:ddop}
\footnotesize
\setlength{\tabcolsep}{6pt}
\begin{tabular}{@{}lrrrrrrr@{}}
\toprule
 & & \multicolumn{4}{c}{Realized cost} & & \\
\cmidrule(lr){3-6}
Method
& \shortstack{Mean\\order}
& Mean (SD)
& \shortstack{Mean\\difference:\\95\% CI}
& $\mathrm{CVaR}_{0.1}$
& \shortstack{$\mathrm{CVaR}_{0.1}$\\difference:\\95\% CI}
& \shortstack{Stockout\\days (\%)}
& \shortstack{Operational\\success (\%)} \\
\midrule
\multicolumn{8}{@{}l}{\emph{Methods returning a certificate}} \\
ODT-JRRO
& 36.9 & 13.8 (17.4) & --
& 51.3 & -- & 21.6 & 81.0 \\
JERO-OLS
& 28.5 & 18.2 (23.9) & $[2.2,7.4]$
& 73.2 & $[13.7,29.8]$ & 63.4 & 67.3 \\
\addlinespace
\multicolumn{8}{@{}l}{\emph{Methods without a certificate}} \\
ODT-SAA
& 34.9 & 13.9 (19.4) & $[-0.5,0.9]$
& 55.9 & $[1.8,8.0]$ & 28.8 & 81.0 \\
ODT-mean
& 29.4 & 17.9 (23.6) & $[1.6,7.0]$
& 72.1 & $[14.0,24.3]$ & 59.5 & 67.3 \\
kNN-SAA
& 31.9 & 15.9 (22.9) & $[0.1,4.8]$
& 68.5 & $[9.3,26.0]$ & 39.9 & 77.8 \\
Global-SAA
& 36.0 & 15.1 (20.5) & $[-0.1,2.9]$
& 61.5 & $[3.3,20.8]$ & 24.2 & 81.0 \\
\bottomrule
\end{tabular}
\end{table}

Table~\ref{tab:ddop} shows that ODT-JRRO achieves the lowest observed
mean cost and $\mathrm{CVaR}_{0.1}$ among the methods considered.
Relative to ODT-SAA, it reduces mean cost by $1.17\%$ and
$\mathrm{CVaR}_{0.1}$ by $8.27\%$, computed from unrounded values.
The mean-cost difference has a confidence interval containing zero,
whereas the interval for the CVaR difference lies strictly above zero.
The evidence of improved performance is therefore clearer for
upper-tail cost than for average cost.
The higher mean order is also accompanied by a $25\%$ reduction
in the number of stockout days.
Operational success is 81.05\% for both methods, so the lower tail cost does not come at the expense of meeting the target less often.
The principal benefit relative to ODT-SAA is greater empirical
robustness to alternative trees. When both methods' orders are evaluated under the same admissible tree class, target, and radius cap,
the mean breakdown threshold increases from $0.018250$ for ODT-SAA
to $0.021453$ for ODT-JRRO, in normalized squared-demand units.
This represents a gain of $17.55\%$ in the mean tolerated excess
estimation loss. The ODT-SAA orders are held fixed when computing
their breakdown thresholds, so this comparison directly measures
the robustness gained by changing the decision.
The threshold strictly increases on $71.24\%$ of test queries
and is unchanged on the others. Against the remaining methods, ODT-JRRO achieves significantly lower mean and tail cost than JERO-OLS, ODT-mean and kNN-SAA, and significantly lower tail cost than Global-SAA. JERO-OLS, the only other method that returns a certificate, has a mean realized cost of 18.2, above the target. Because it robustifies the conditional-mean prediction without accounting for residual dispersion, it places low orders and stocks out on 63.4\% of days.

\subsection{Interpretable outcome on real-world data}
We illustrate how the proposed method explains a local decision, using the setting of
Section~6.2 with the first held-out date, 8 June 2015, as the query. Its observed demand is used
neither to construct the decision nor to set the target. Order quantities and costs are reported
in original units, while tree losses and loss tolerances use demand normalized by its training
mean.
The resulting order is $q^\star\approx36.225$, compared with the nominal-tree order of $32$, an increase of approximately $13.2\%$. The cost target is $\tau\approx18.0719$. Algorithm~\ref{alg:dp} retains five neighborhoods, denoted $S_1,\ldots,S_5$, and identifies the breakdown threshold $r^\star\approx0.019232$. The returned order satisfies the target for every radius strictly below $r^\star$, while no order is feasible at $r^\star$.

Figure~\ref{fig:interpret-evolution} shows how the admissible neighborhood family expands and the feasible-order interval contracts as the radius increases. For this illustration, the full family is obtained by complete enumeration through the breakdown threshold. The upper panel reports $|\mathcal S(x^0,r)|$, that is, the number of distinct neighborhoods, for a total of $26940$ neighborhoods including $S_5$. The lower panel shows the corresponding feasible-order intervals. As the radius increases, newly admitted neighborhoods may raise the lower bound or reduce the upper bound, while others leave both unchanged. Immediately before breakdown, $q^\star$ lies at the lower endpoint of the remaining interval. At $r^\star$, the admission of $S_5$ makes the interval empty.
\begin{figure}[tb]
    \centering
    \includegraphics[width=0.8\linewidth]
    {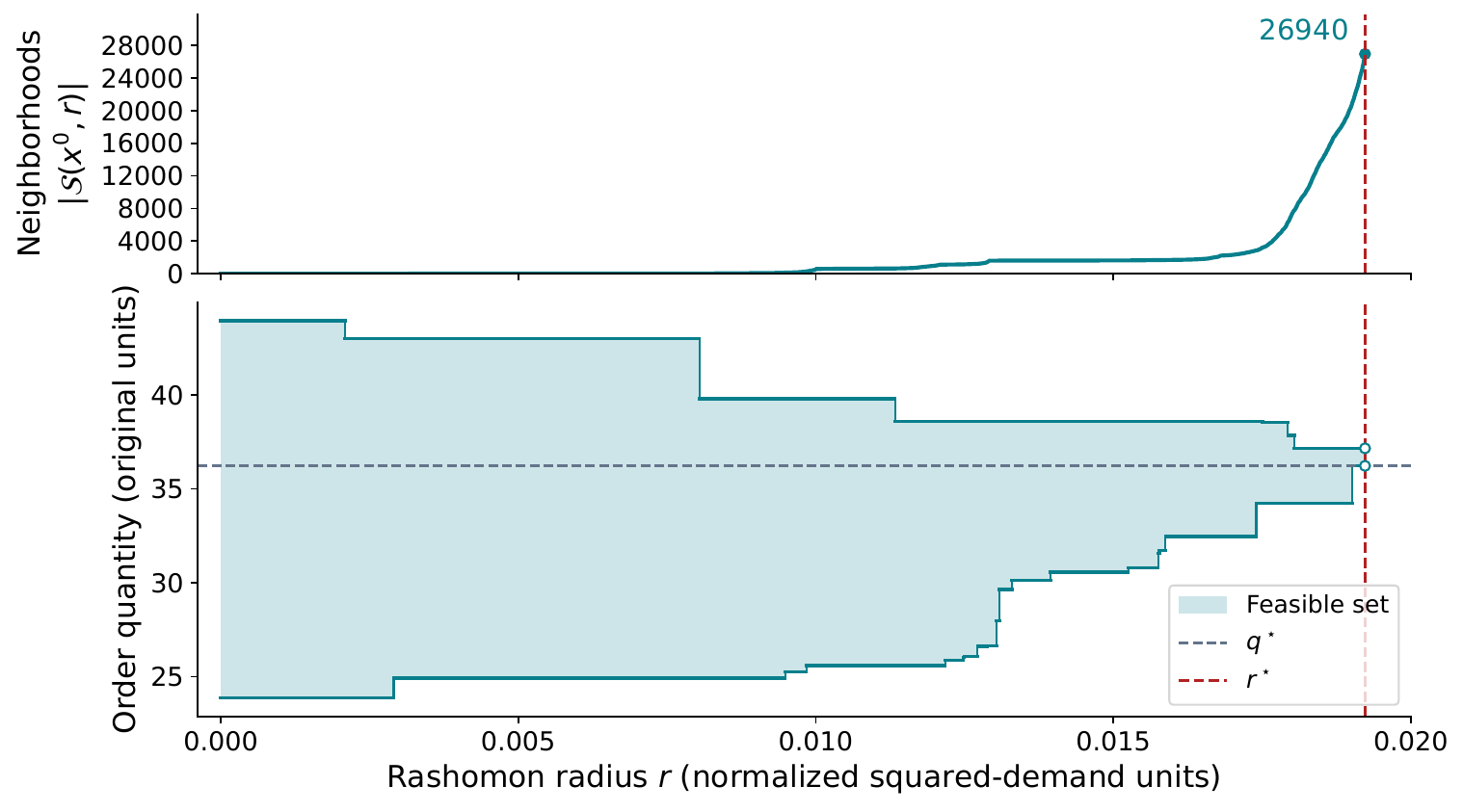}
    \caption{Expansion of the admissible neighborhood family (upper panel) and contraction of the feasible-order interval (lower panel). The full family contains $26939$ neighborhoods strictly below the breakdown threshold and $26940$ at $r^\star$, marked by the vertical dashed line. The shaded interval contains orders satisfying all admitted constraints, and the horizontal dashed line marks $q^\star$. The feasible set is empty at $r^\star$.}
    \label{fig:interpret-evolution}
\end{figure}

Figure~\ref{fig:interpret-membership} shows the training observations belonging to each retained neighborhood, together with its size, entry level, and feasible-order interval. These neighborhoods explain both the change in the order and the limit of robustness. The nominal neighborhood $S_1$ supports the order of $32$, while $S_4$ requires an order of approximately $36.22$ or more, explaining the increase to $q^\star$. The neighborhood $S_5$ has a different role: its feasible-order interval is empty, so no order can satisfy the target once its entry level is reached.
At the displayed precision, the constraints from $S_1,\ldots,S_4$ alone leave the interval $[36.22,42.44]$. Additional neighborhoods in the full family lower the upper endpoint further before breakdown. This distinction reflects the role of constraint generation: separation certifies the returned decision against the full admissible family, while the retained constraints need not reproduce every boundary of the full feasible set.
\begin{figure}[tb]
    \centering
    \includegraphics[width=\linewidth]
    {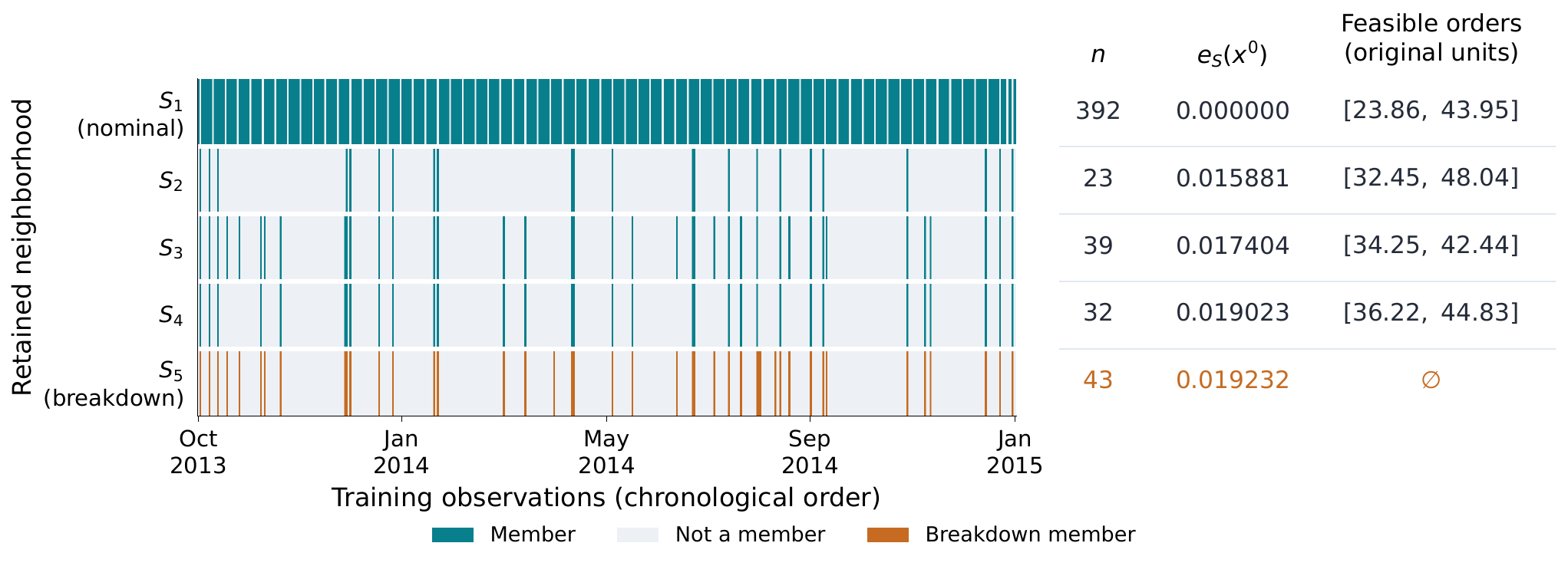}
    \caption{Membership of chronologically ordered training observations in the five retained neighborhoods. The right-hand columns report neighborhood size, entry level, and feasible-order interval.}
    \label{fig:interpret-membership}
\end{figure}
To explain why admitting $S_5$ makes the target constraint infeasible, write the empirical cost on a neighborhood $S$ as
\(
    \widehat C_S(q)
    :=\frac{1}{|S|}\sum_{i\in S}c(q,y_i).
\)
The minimum tree loss compatible with $S_5$ is approximately $0.142377$, compared with the nominal optimum $\rho^\star\approx0.123145$. Its entry level is therefore
\(
    r^\star
    =e_{S_5}(x^0)
    =\rho_{S_5}(x^0)-\rho^\star
    \approx0.019232.
\)
At the returned order, $\widehat C_{S_5}(q^\star)\approx18.5320$, which exceeds the target. Moreover, an optimal order for this neighborhood alone is $39$, with
\(
    \min_{q\in\mathcal Z}\widehat C_{S_5}(q)
    \approx18.3953
    >\tau.
\)
Hence, changing the order cannot restore feasibility once $S_5$ becomes admissible. 

As shown in Figure~\ref{fig:interpret-breakdown}, a single neighborhood is sufficient in this instance to establish breakdown for every decision. The nominal order \(q^0\) already violates the target on \(S_2\). For this query, the \(13.2\%\) increase in order quantity raises the nominal empirical cost on \(S_1\) by \(6.19\%\) and increases the breakdown threshold by \(21.10\%\), quantifying the cost and robustness trade-off.
\begin{figure}[tb]
    \centering
    \includegraphics[width=0.6\linewidth]
    {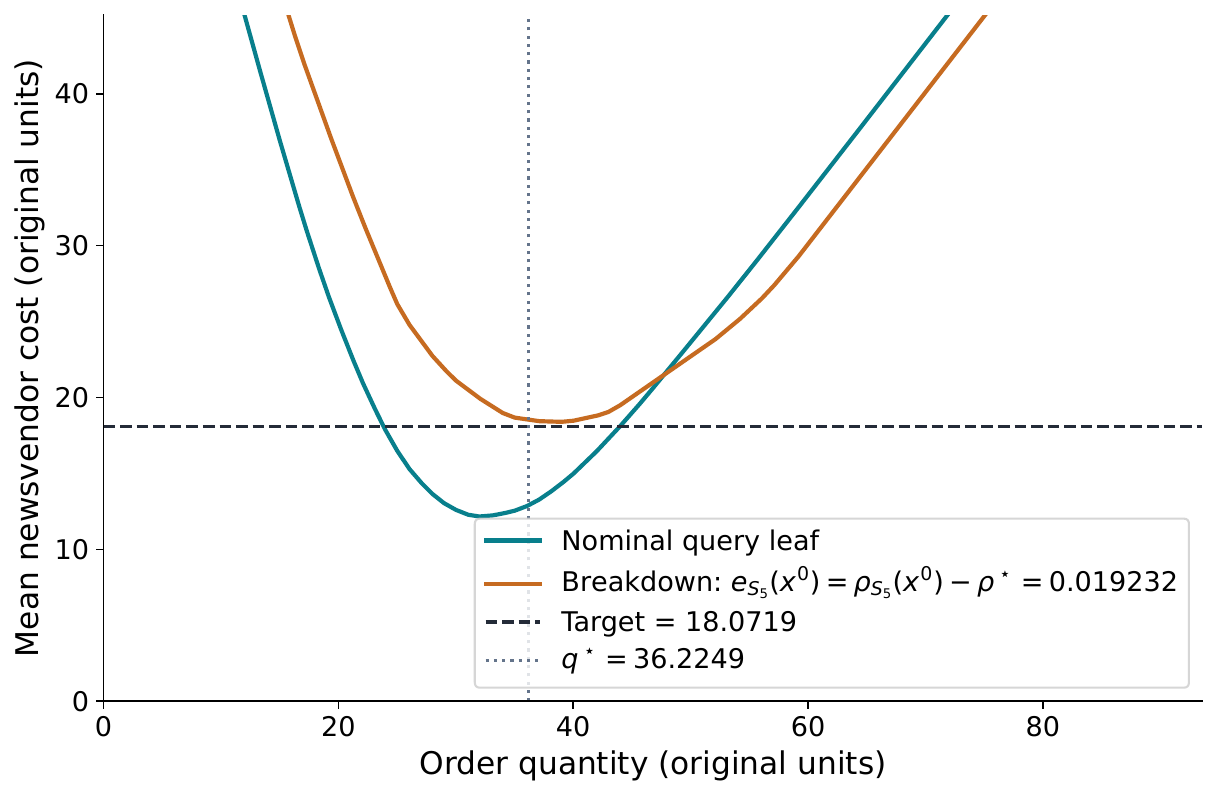}
    \caption{Empirical mean cost as a function of the order for the nominal and breakdown neighborhoods. The horizontal dashed line is the target and the vertical dotted line is $q^\star$. The minimum cost on the breakdown neighborhood exceeds the target, so no order can preserve feasibility when this neighborhood enters.}
    \label{fig:interpret-breakdown}
\end{figure}
Figure~\ref{fig:interpret-trees} shows the nominal tree and a minimum-loss tree inducing $S_5$. The highlighted query paths identify the partitions underlying the nominal and breakdown neighborhoods. The interpretation is local to the selected query and draws on several trees and neighborhoods, and it does not provide a single policy tree representing the robust decisions for all possible queries.

\begin{figure}[tb]
    \centering
    \begin{minipage}[c]{0.4\linewidth}
        \centering
        \includegraphics[width=\linewidth]
        {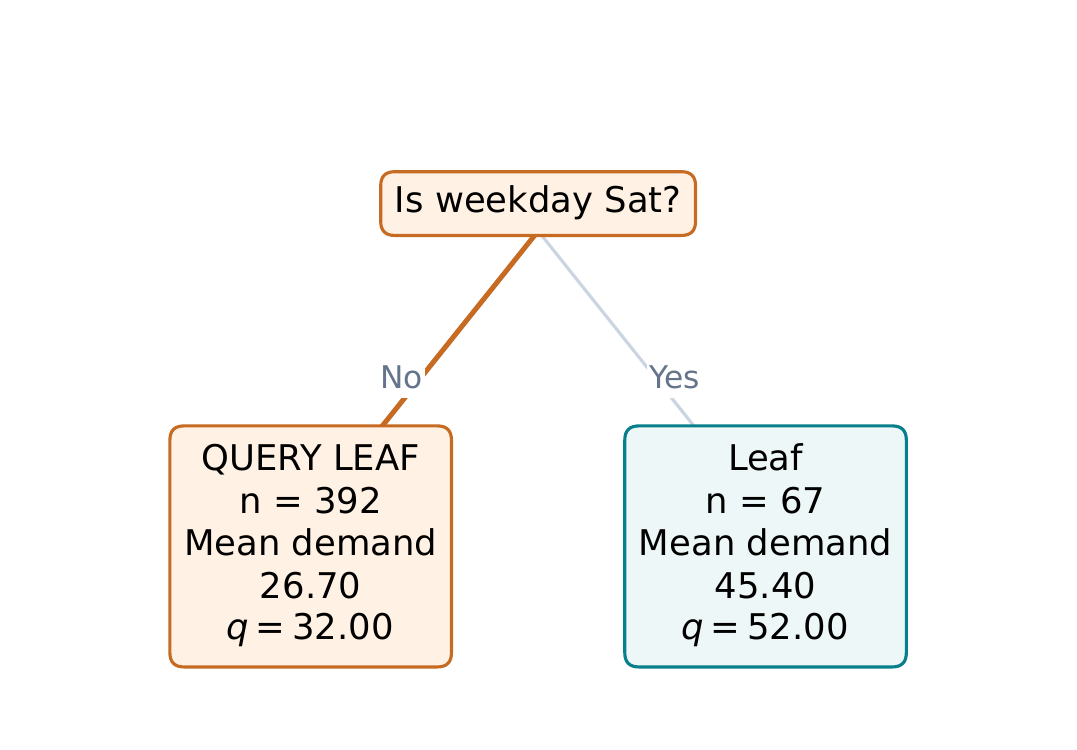}
    \end{minipage}
    \hfill
    \begin{minipage}[c]{0.4\linewidth}
        \centering
        \includegraphics[width=\linewidth]
        {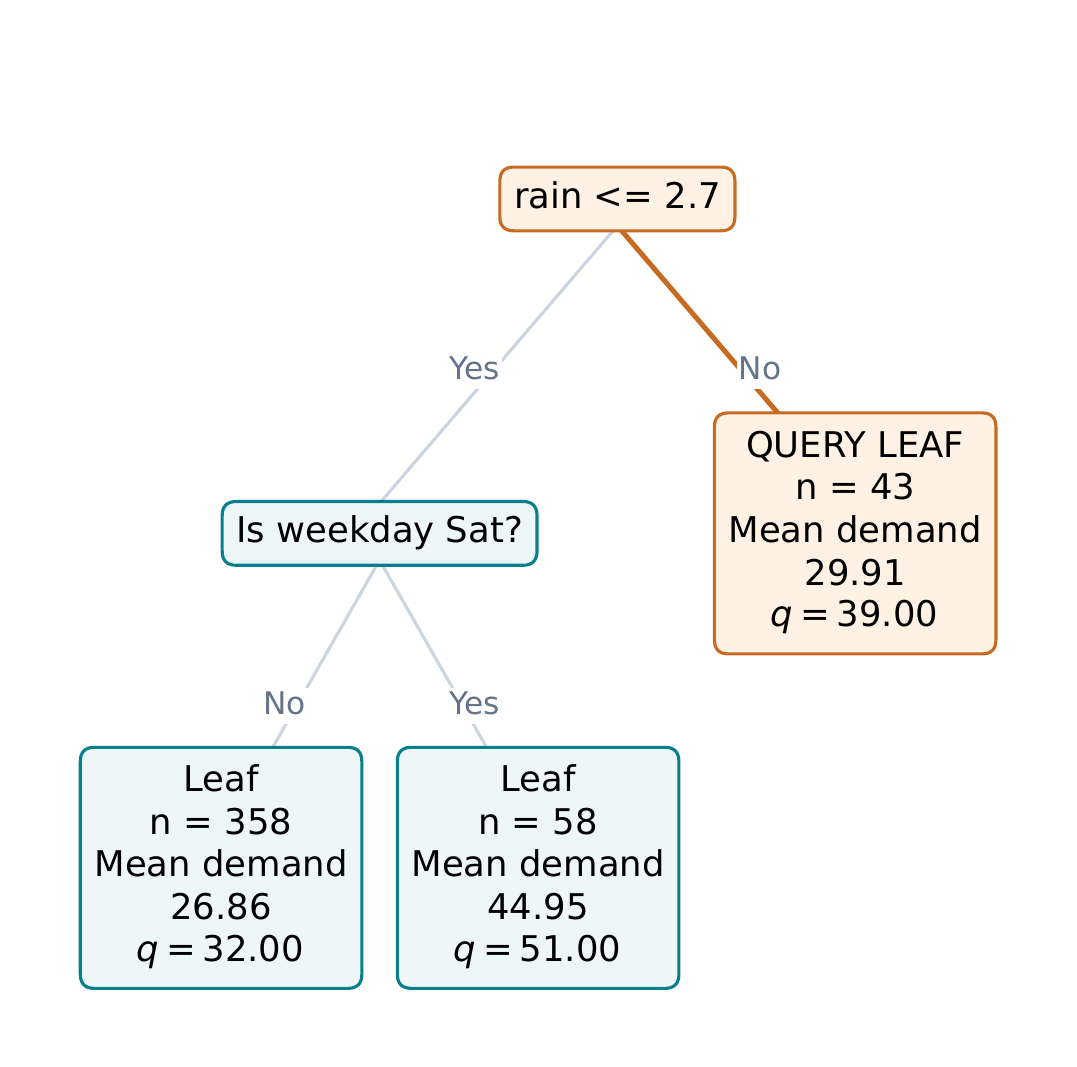}
    \end{minipage}
    \caption{Nominal tree (left) and a minimum-loss tree inducing the breakdown neighborhood $S_5$ (right). The highlighted paths route the query to its terminal leaf. Each leaf reports its sample size, mean demand, and empirical newsvendor order.}
    \label{fig:interpret-trees}
\end{figure}
We next examine how the minimum leaf size affects these results, keeping the training sample, query, target, maximum depth, and split penalty fixed. Figure~\ref{fig:nmin-sensitivity} compares the resulting breakdown thresholds with the CVaR envelope of Proposition~\ref{prop:cvar}.
Among the tested minimum leaf sizes, robustness over the full attainable tree class first occurs at $N_{\min}=161$, a fraction of $161/459\approx0.351$. At this point, however, even the minimum CVaR remains above the target. Thus, no order satisfies the sufficient feasibility condition supplied by the CVaR bound. This difference illustrates the conservatism introduced by considering arbitrary subsets of at least $N_{\min}$ observations, including subsets that cannot be induced as query neighborhoods by an admissible tree. The CVaR condition is first satisfied only at $N_{\min}=414$, a fraction of $414/459\approx0.902$. For $N_{\min}/N>1/2$, no split is feasible and the tree class already contains only the root leaf. As anticipated in Section \ref{sec:min-leaf-envelope}, the finite thresholds in Figure~\ref{fig:nmin-sensitivity} are not monotone in $N_{\min}$, because $\rho^\star$ changes with the tree class.

\begin{figure}[tbp]
    \centering
    \includegraphics[width=0.9\linewidth]
    {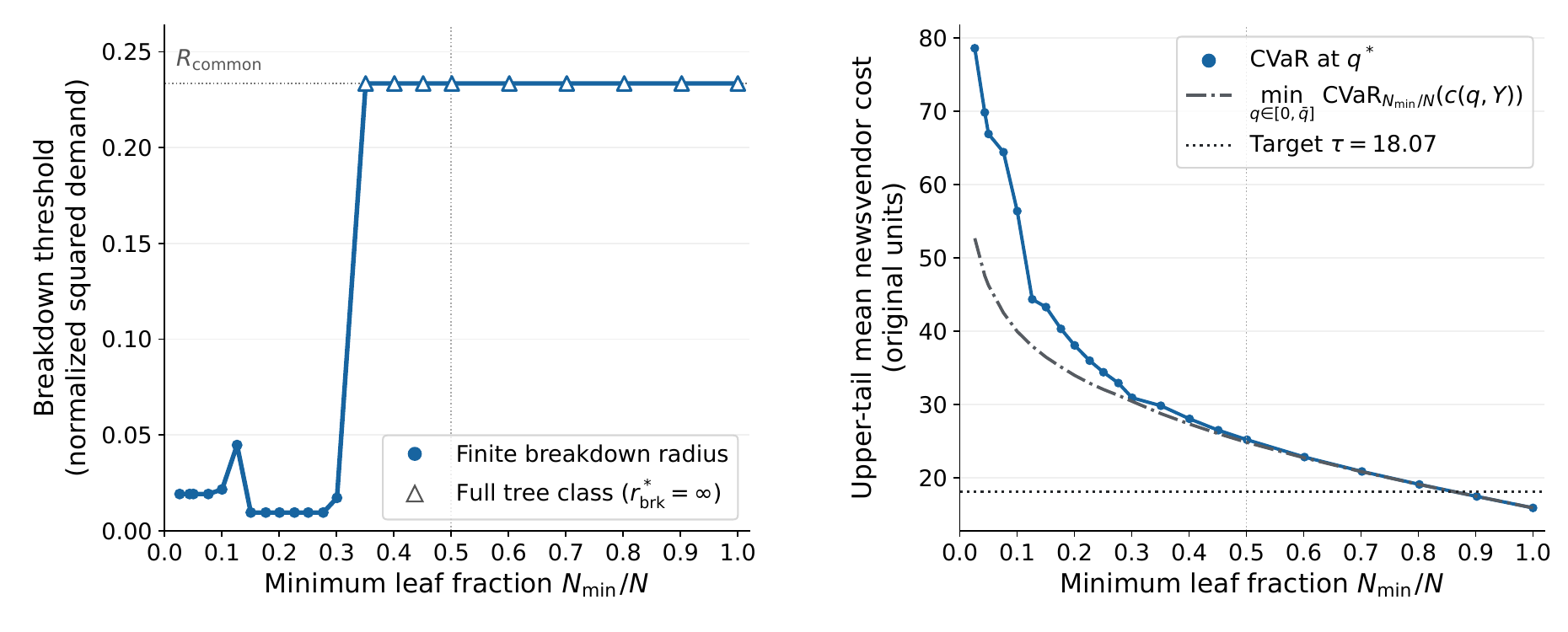}
    \caption{Sensitivity to the minimum leaf size. Left: finite
breakdown thresholds (circles) and robustness over the entire admissible
tree class (triangles). Full-class cases have infinite uncapped breakdown
thresholds and are displayed at the common ceiling
$R_{\mathrm{common}}=V_{\mathrm{norm}}+\lambda(2^D-1)\approx0.2335$,
where $V_{\mathrm{norm}}$ is the empirical variance of normalized training
demand.
Right: empirical upper-tail CVaR of the newsvendor cost at the returned
order $q^\star$, its minimum over $q\in\mathcal Z$, and the target $\tau$,
using tail fraction $N_{\min}/N$. The vertical dotted lines mark $N_{\min}/N=1/2$, above which the only feasible tree consists of a single leaf at the root.}
    \label{fig:nmin-sensitivity}
\end{figure}

\section{Conclusion and future work}\label{sec:conclusion}
This paper extends joint estimation and robustness optimization to regression trees and the local empirical conditional distributions they induce. The resulting ODT-JRRO model jointly optimizes an operational decision and the estimation-loss tolerance defining a Rashomon set, subject to prescribed targets holding under every induced distribution. The formulation links predictive multiplicity to contextual optimization by identifying the largest set of near-optimal models under which a common decision satisfies the prescribed targets.
The central structural result is that trees sharing the same query neighborhood are indistinguishable to the downstream constraints. Assigning each attainable neighborhood its minimum compatible estimation loss therefore provides a finite representation of the relevant uncertainty. We use the resulting entry levels to characterize the boundary of the feasible radii and develop a constraint generation algorithm with dynamic programming for query-path pricing. This representation avoids enumerating the individual trees and distinguishes between an attained radius cap and an excluded breakdown endpoint.
The computational experiments show that the algorithm typically requires
few neighborhoods to certify an outcome. On depth-two instances, the algorithm achieves speedups of
approximately $7 \times$ to $39 \times$ relative to full neighborhood enumeration, and our pruning mechanism substantially reduces
the number of subtree states evaluated. At larger depths, computing the
nominal optimum remains the main obstacle to scalability. 
The benchmark on real-world restaurant demand illustrates the operational implications
of optimizing robustness. Relative to ODT-SAA, ODT-JRRO increases the mean
breakdown threshold by 17.55\% and reduces empirical
$\operatorname{CVaR}_{0.1}$ by 8.27\%. The mean-cost improvement is not
statistically significant, and both methods meet the realized-cost target
on 81.05\% of test days. These results support greater robustness to alternative near-optimal trees and lower observed tail cost.
The empirical comparison assumes that realized same-day weather is observed. Evaluating orders based on weather forecasts available at the decision time remains future work.
The interpretability case study shows how individual neighborhoods explain both the
change in the order and the limit of robustness. The minimum leaf size
analysis also illustrates the conservatism of the CVaR envelope, which
includes subsets that cannot arise as admissible query neighborhoods.
The robustness certificates returned by ODT-JRRO are conditional on the observed data and the chosen tree class.
An important next step is to establish when such a certificate also implies out-of-sample target satisfaction, distinguishing the effects of sampling error from those of misspecification of the predictor class.
In problems with endogenous uncertainty, such as joint pricing and inventory optimization, the decision influences the distribution of demand~\citep{bertsimas2020predictive}. Extending the framework to this setting would require assumptions that
identify the effect of decisions on demand from historical data, together
with a neighborhood representation and separation procedure that
accommodate decision-dependent query routing.

\section*{Acknowledgements}
This work was supported by the European Research Council (ERC) under the European Union’s Horizon Europe programme (grant agreement No. 101093188, DECIDE).

\section*{CRediT authorship contribution statement}

\textbf{Lorenzo Bonasera:} Conceptualization, Methodology, Software, Validation, Investigation, Visualization, Writing- original draft, Writing-  review and editing.

\textbf{David Pisinger:} Supervision, Validation, Resources, Funding acquisition, Project administration, Writing- review and editing.

\section*{Declaration of generative AI and AI-assisted technologies in the manuscript preparation process}
During the preparation of this work the authors used Claude to assist with language editing and technical consistency checks. After using this tool, the authors reviewed and edited the content as needed and take full responsibility for the content of the published article.

%\scriptsize
\bibliographystyle{elsarticle-harv}
\bibliography{references}

@article{zhu2022joint,
  title={Joint estimation and robustness optimization},
  author={Zhu, Taozeng and Xie, Jingui and Sim, Melvyn},
  journal={Management Science},
  volume={68},
  number={3},
  pages={1659--1677},
  year={2022},
  publisher={INFORMS}
}

@inproceedings{hsu2022rashomon,
  title={{R}ashomon capacity: A metric for predictive multiplicity in classification},
  author={Hsu, Hsiang and Calmon, Flavio P.},
  booktitle={Advances in Neural Information Processing Systems},
  volume={35},
  pages={28988--29000},
  year={2022}
}

@inproceedings{marx2020predictive,
  title     = {Predictive multiplicity in classification},
  author    =  {Marx, Charles and Calmon, Flavio P. and Ustun, Berk},
  booktitle = {Proceedings of the 37th International Conference on Machine Learning},
  series    = {Proceedings of Machine Learning Research},
  volume    = {119},
  pages     = {6765--6774},
  publisher = {PMLR},
  year      = {2020}
}

@inproceedings{semenova2022existence,
  title={On the existence of simpler machine learning models},
  author={Semenova, Lesia and Rudin, Cynthia and Parr, Ronald},
  booktitle={Proceedings of the 2022 ACM Conference on Fairness, Accountability, and Transparency},
  pages={1827--1858},
  year={2022}
}

@article{fisher2019all,
  title={All models are wrong, but many are useful: Learning a variable's importance by studying an entire class of prediction models simultaneously},
  author={Fisher, Aaron and Rudin, Cynthia and Dominici, Francesca},
  journal={Journal of Machine Learning Research},
  volume={20},
  number={177},
  pages={1--81},
  year={2019}
}

@article{breiman2001statistical,
  title={Statistical modeling: The two cultures (with comments and a rejoinder by the author)},
  author={Breiman, Leo},
  journal={Statistical Science},
  volume={16},
  number={3},
  pages={199--231},
  year={2001},
  publisher={Institute of Mathematical Statistics}
}

@article{demirovic2022murtree,
  title={{MurTree}: Optimal decision trees via dynamic programming and search},
  author={Demirovi{\'c}, Emir and Lukina, Anna and Hebrard, Emmanuel and Chan, Jeffrey and Bailey, James and Leckie, Christopher and Ramamohanarao, Kotagiri and Stuckey, Peter J.},
  journal={Journal of Machine Learning Research},
  volume={23},
  number={26},
  pages={1--47},
  year={2022}
}

@inproceedings{lin2020generalized,
  title     = {Generalized and scalable optimal sparse decision trees},
  author    = {Lin, Jimmy and Zhong, Chudi and Hu, Diane and Rudin, Cynthia and Seltzer, Margo},
  booktitle = {Proceedings of the 37th International Conference on Machine Learning},
  series    = {Proceedings of Machine Learning Research},
  volume    = {119},
  pages     = {6150--6160},
  publisher = {PMLR},
  year      = {2020}
}

@inproceedings{aglin2020learning,
  title     = {Learning optimal decision trees using caching branch-and-bound search},
  author    = {Aglin, Ga{\"e}l and Nijssen, Siegfried and Schaus, Pierre},
  booktitle = {Proceedings of the AAAI Conference on Artificial Intelligence},
  volume    = {34},
  pages     = {3146--3153},
  year      = {2020}
}

@article{long2023robust,
  title={Robust satisficing},
  author={Long, Daniel Zhuoyu and Sim, Melvyn and Zhou, Minglong},
  journal={Operations Research},
  volume={71},
  number={1},
  pages={61--82},
  year={2023},
  publisher={INFORMS}
}

@article{brown2009satisficing,
  title={Satisficing measures for analysis of risky positions},
  author={Brown, David B. and Sim, Melvyn},
  journal={Management Science},
  volume={55},
  number={1},
  pages={71--84},
  year={2009},
  publisher={INFORMS}
}

@article{rockafellar2000optimization,
  title   = {Optimization of conditional value-at-risk},
  author  = {Rockafellar, R. Tyrrell and Uryasev, Stanislav},
  journal = {Journal of Risk},
  volume  = {2},
  number  = {3},
  pages   = {21--41},
  year    = {2000}
}

@inproceedings{sparseodt,
  title     = {Optimal sparse decision trees},
  author    = {Hu, Xiyang and Rudin, Cynthia and Seltzer, Margo},
  booktitle = {Advances in Neural Information Processing Systems},
  volume    = {32},
  pages     = {7267--7275},
  year      = {2019}
}

@article{aghaei2025strong,
  title={Strong optimal classification trees},
  author={Aghaei, Sina and G{\'o}mez, Andr{\'e}s and Vayanos, Phebe},
  journal={Operations Research},
  volume={73},
  number={4},
  pages={2223--2241},
  year={2025},
  publisher={INFORMS}
}

@article{categoricalodt,
  title={Optimal decision trees for categorical data via integer programming},
  author={G{\"u}nl{\"u}k, Oktay and Kalagnanam, Jayant and Li, Minhan and Menickelly, Matt and Scheinberg, Katya},
  journal={Journal of Global Optimization},
  volume={81},
  number={1},
  pages={233--260},
  year={2021},
  publisher={Springer}
}

@article{bertsimas2019optimal,
  title={Optimal prescriptive trees},
  author={Bertsimas, Dimitris and Dunn, Jack and Mundru, Nishanth},
  journal={INFORMS Journal on Optimization},
  volume={1},
  number={2},
  pages={164--183},
  year={2019},
  publisher={INFORMS}
}

@article{bertsimas2020predictive,
  title={From predictive to prescriptive analytics},
  author={Bertsimas, Dimitris and Kallus, Nathan},
  journal={Management Science},
  volume={66},
  number={3},
  pages={1025--1044},
  year={2020},
  publisher={INFORMS}
}

@inproceedings{xin2022exploring,
  title     = {Exploring the whole {R}ashomon set of sparse decision trees},
  author    = {Xin, Rui and Zhong, Chudi and Chen, Zhi and Takagi, Takuya and Seltzer, Margo and Rudin, Cynthia},
  booktitle = {Advances in Neural Information Processing Systems},
  volume    = {35},
  pages     = {14071--14084},
  year      = {2022}
}

@article{kallus2023stochastic,
  title={Stochastic optimization forests},
  author={Kallus, Nathan and Mao, Xiaojie},
  journal={Management Science},
  volume={69},
  number={4},
  pages={1975--1994},
  year={2023},
  publisher={INFORMS}
}

@article{levatic2018semi,
  title={Semi-supervised trees for multi-target regression},
  author={Levati{\'c}, Jurica and Kocev, Dragi and Ceci, Michelangelo and D{\v{z}}eroski, Sa{\v{s}}o},
  journal={Information Sciences},
  volume={450},
  pages={109--127},
  year={2018},
  publisher={Elsevier}
}

@article{boutilier2023optimal,
  title={Optimal multivariate decision trees},
  author={Boutilier, Justin J. and Michini, Carla and Zhou, Zachary},
  journal={Constraints},
  volume={28},
  number={4},
  pages={549--577},
  year={2023},
  publisher={Springer}
}

@inproceedings{boutilier2022shattering,
  title     = {Shattering inequalities for learning optimal decision trees},
  author    = {Boutilier, Justin J. and Michini, Carla and Zhou, Zachary},
  booktitle = {Integration of Constraint Programming, Artificial Intelligence, and Operations Research (CPAIOR 2022)},
  series    = {Lecture Notes in Computer Science},
  volume    = {13292},
  pages     = {74--90},
  year      = {2022},
  publisher = {Springer}
}

@article{blanquero2021optimal,
  title={Optimal randomized classification trees},
  author={Blanquero, Rafael and Carrizosa, Emilio and Molero-R{\'\i}o, Cristina and Romero Morales, Dolores},
  journal={Computers \& Operations Research},
  volume={132},
  pages={105281},
  year={2021},
  publisher={Elsevier}
}

@inproceedings{van2024piecewise,
  title     = {Piecewise constant and linear regression trees: An optimal dynamic programming approach},
  author    = {{van den Bos}, Mim and van der Linden, Jacobus G. M. and Demirovi{\'c}, Emir},
  booktitle = {Proceedings of the 41st International Conference on Machine Learning},
  volume    = {235},
  pages     = {48994--49007},
  year      = {2024}
}

@inproceedings{liu2021risk,
  title={Risk bounds and calibration for a smart predict-then-optimize method},
  author={Liu, Heyuan and Grigas, Paul},
  booktitle={Advances in Neural Information Processing Systems},
  volume={34},
  pages={22083--22094},
  year={2021}
}

@article{el2023generalization,
  title     = {Generalization bounds in the predict-then-optimize framework},
  author    = {El Balghiti, Othman and Elmachtoub, Adam N. and Grigas, Paul and Tewari, Ambuj},
  journal   = {Mathematics of Operations Research},
  volume    = {48},
  number    = {4},
  pages     = {2043--2065},
  year      = {2023},
  publisher = {INFORMS}
}

@article{mandi2024decision,
  title={Decision-focused learning: Foundations, state of the art, benchmark and future opportunities},
  author={Mandi, Jayanta and Kotary, James and Berden, Senne and Mulamba, Maxime and Bucarey, Victor and Guns, Tias and Fioretto, Ferdinando},
  journal={Journal of Artificial Intelligence Research},
  volume={80},
  pages={1623--1701},
  year={2024}
}

@article{bonasera2024optimal,
  title={Optimal shapelets tree for time series interpretable classification},
  author={Bonasera, Lorenzo and Gualandi, Stefano},
  journal={EURO Journal on Computational Optimization},
  volume={12},
  pages={100091},
  year={2024},
  publisher={Elsevier}
}

@inproceedings{zhang2023optimal,
  title     = {Optimal sparse regression trees},
  author    = {Zhang, Rui and Xin, Rui and Seltzer, Margo and Rudin, Cynthia},
  booktitle = {Proceedings of the AAAI Conference on Artificial Intelligence},
  volume    = {37},
  pages     = {11270--11279},
  year      = {2023}
}

@article{blanquero2022sparse,
  title={On sparse optimal regression trees},
  author={Blanquero, Rafael and Carrizosa, Emilio and Molero-R{\'\i}o, Cristina and Romero Morales, Dolores},
  journal={European Journal of Operational Research},
  volume={299},
  number={3},
  pages={1045--1054},
  year={2022},
  publisher={Elsevier}
}

@inproceedings{verwer2017learning,
  title     = {Learning decision trees with flexible constraints and objectives using integer optimization},
  author    = {Verwer, Sicco and Zhang, Yingqian},
  booktitle = {Integration of AI and OR Techniques in Constraint Programming (CPAIOR 2017)},
  series    = {Lecture Notes in Computer Science},
  volume    = {10335},
  pages     = {94--103},
  year      = {2017},
  publisher = {Springer}
}

@inproceedings{van2023necessary,
  title={Necessary and sufficient conditions for optimal decision trees using dynamic programming},
  author={{van der Linden}, Jacobus G. M. and de Weerdt, Mathijs and Demirovi{\'c}, Emir},
  booktitle={Advances in Neural Information Processing Systems},
  volume={36},
  pages={9173--9212},
  year={2023}
}

@article{hyafil1976constructing,
title = {Constructing optimal binary decision trees is {NP}-complete},
journal = {Information Processing Letters},
volume = {5},
number = {1},
pages = {15--17},
year = {1976},
author = {Hyafil, Laurent and Rivest, Ronald L.},
}

@article{philippi2021ddop,
  title={ddop: A {P}ython package for data-driven operations management},
  author={Philippi, Andreas and Buttler, Simone and Stein, Nikolai},
  journal={Journal of Open Source Software},
  volume={6},
  number={66},
  pages={3429},
  year={2021}
}

@book{c4.5,
author = {Quinlan, J. Ross},
title = {{C4.5}: Programs for machine learning},
year = {1993},
isbn = {1558602380},
publisher = {Morgan Kaufmann},
address = {San Mateo, CA}
}

@book{breiman1984classification,
  author    = {Breiman, Leo and Friedman, Jerome H. and
               Olshen, Richard A. and Stone, Charles J.},
  title     = {Classification and regression trees},
  publisher = {Wadsworth International Group},
  address   = {Belmont, CA},
  year      = {1984}
}

@inproceedings{black2022model,
  title     = {Model multiplicity: Opportunities, concerns, and solutions},
  author    = {Black, Emily and Raghavan, Manish and Barocas, Solon},
  booktitle = {Proceedings of the 2022 ACM Conference on Fairness, Accountability, and Transparency},
  pages     = {850--863},
  year      = {2022}
}

@article{creel2022algorithmic,
  title={The algorithmic leviathan: Arbitrariness, fairness, and opportunity in algorithmic decision-making systems},
  author={Creel, Kathleen and Hellman, Deborah},
  journal={Canadian Journal of Philosophy},
  volume={52},
  number={1},
  pages={26--43},
  year={2022},
  publisher={Cambridge University Press}
}

@article{rudin2019stop,
  title={Stop explaining black box machine learning models for high stakes decisions and use interpretable models instead},
  author={Rudin, Cynthia},
  journal={Nature Machine Intelligence},
  volume={1},
  number={5},
  pages={206--215},
  year={2019},
  publisher={Springer Nature}
}

@inproceedings{hsu2024rashomongb,
  title={{RashomonGB}: Analyzing the {Rashomon} effect and mitigating predictive multiplicity in gradient boosting},
  author={Hsu, Hsiang and Brugere, Ivan and Sharma, Shubham and L{\'e}cu{\'e}, Freddy and Chen, Chun-Fu},
  booktitle={Advances in Neural Information Processing Systems},
  volume={37},
  pages={121265--121303},
  year={2024}
}

@inproceedings{paes2023inevitability,
  title     = {On the inevitability of the {R}ashomon effect},
  author    = {Paes, Lucas Monteiro and Cruz, Rodrigo and Calmon, Flavio P. and Diaz, Mario},
  booktitle = {2023 IEEE International Symposium on Information Theory (ISIT)},
  pages     = {549--554},
  year      = {2023},
  publisher = {IEEE}
}

@article{carrizosa2021mathematical,
  title     = {Mathematical optimization in classification and regression trees},
  author    = {Carrizosa, Emilio and Molero-R{\'\i}o, Cristina and Romero Morales, Dolores},
  journal   = {TOP},
  volume    = {29},
  number    = {1},
  pages     = {5--33},
  year      = {2021},
  publisher = {Springer}
}

@article{ben1998robust,
  title={Robust convex optimization},
  author={Ben-Tal, Aharon and Nemirovski, Arkadi},
  journal={Mathematics of Operations Research},
  volume={23},
  number={4},
  pages={769--805},
  year={1998},
  publisher={INFORMS}
}

@article{soyster1973convex,
  title={Convex programming with set-inclusive constraints and applications to inexact linear programming},
  author={Soyster, Allen L.},
  journal={Operations Research},
  volume={21},
  number={5},
  pages={1154--1157},
  year={1973},
  publisher={INFORMS}
}

@article{wiesemann2014distributionally,
  title={Distributionally robust convex optimization},
  author={Wiesemann, Wolfram and Kuhn, Daniel and Sim, Melvyn},
  journal={Operations Research},
  volume={62},
  number={6},
  pages={1358--1376},
  year={2014},
  publisher={INFORMS}
}

@article{bertsimas2004price,
  title     = {The price of robustness},
  author    = {Bertsimas, Dimitris and Sim, Melvyn},
  journal   = {Operations Research},
  volume    = {52},
  number    = {1},
  pages     = {35--53},
  year      = {2004},
  publisher = {INFORMS}
}

@article{kannan2024residuals,
  title   = {Residuals-based distributionally robust optimization with covariate information},
  author  = {Kannan, Rohit and Bayraksan, G{\"u}zin and Luedtke, James R.},
  journal = {Mathematical Programming},
  volume  = {207},
  number  = {1},
  pages   = {369--425},
  year    = {2024}
}

@inproceedings{arslan2025sorted,
  title     = {{SORTeD} {R}ashomon sets of sparse decision trees: Anytime enumeration},
  author    = {Arslan, Elif and van der Linden, Jacobus G. M. and Hoogendoorn, Serge and Rinaldi, Marco and Demirovi{\'c}, Emir},
  booktitle = {Advances in Neural Information Processing Systems},
  volume    = {38},
  pages     = {44717--44755},
  year      = {2025}
}

@article{ribeiro2026integrated,
  title   = {Integrated estimate-and-optimize decision trees learning for two-stage linear decision-making problems},
  author  = {Ribeiro, Rafaela and Fanzeres, Bruno},
  journal = {European Journal of Operational Research},
  volume  = {329},
  number  = {2},
  pages   = {607--628},
  year    = {2026}
}

@article{bertsimas2022survival,
  title={Optimal survival trees},
  author={Bertsimas, Dimitris and Dunn, Jack and Gibson, Emma and Orfanoudaki, Agni},
  journal={Machine Learning},
  volume={111},
  number={8},
  pages={2951--3023},
  year={2022},
  publisher={Springer}
}

@article{amram2022policy,
  title={Optimal policy trees},
  author={Amram, Maxime and Dunn, Jack and Zhuo, Ying Daisy},
  journal={Machine Learning},
  volume={111},
  number={7},
  pages={2741--2768},
  year={2022},
  publisher={Springer}
}

@article{kuhn2025distributionally,
  title={Distributionally robust optimization},
  author={Kuhn, Daniel and Shafiee, Soroosh and Wiesemann, Wolfram},
  journal={Acta Numerica},
  volume={34},
  pages={579--804},
  year={2025},
  publisher={Cambridge University Press}
}

@article{bertsimas2011theory,
  title={Theory and applications of robust optimization},
  author={Bertsimas, Dimitris and Brown, David B. and Caramanis, Constantine},
  journal={SIAM Review},
  volume={53},
  number={3},
  pages={464--501},
  year={2011},
  publisher={SIAM}
}

@article{rahimian2022frameworks,
     author = {Rahimian, Hamed and Mehrotra, Sanjay},
     title = {Frameworks and results in distributionally robust optimization},
     journal = {Open Journal of Mathematical Optimization},
     eid = {4},
     pages = {1--85},
     year = {2022},
     volume = {3}
}

@inproceedings{chenreddy2022data,
  title={Data-driven conditional robust optimization},
  author={Chenreddy, Abhilash Reddy and Bandi, Nymisha and Delage, Erick},
  booktitle={Advances in Neural Information Processing Systems},
  volume={35},
  pages={9525--9537},
  year={2022}
}

@article{mohajerin2018data,
  title     = {Data-driven distributionally robust optimization using the {W}asserstein metric: Performance guarantees and tractable reformulations},
  author    = {Mohajerin Esfahani, Peyman and Kuhn, Daniel},
  journal   = {Mathematical Programming},
  volume    = {171},
  number    = {1},
  pages     = {115--166},
  year      = {2018},
  publisher = {Springer}
}

@article{delage2010distributionally,
  title={Distributionally robust optimization under moment uncertainty with application to data-driven problems},
  author={Delage, Erick and Ye, Yinyu},
  journal={Operations Research},
  volume={58},
  number={3},
  pages={595--612},
  year={2010},
  publisher={INFORMS}
}

@article{deng2022predictive,
  title={Predictive stochastic programming},
  author={Deng, Yunxiao and Sen, Suvrajeet},
  journal={Computational Management Science},
  volume={19},
  number={1},
  pages={65--98},
  year={2022},
  publisher={Springer}
}

@article{liu2021time,
  title={On-time last-mile delivery: Order assignment with travel-time predictors},
  author={Liu, Sheng and He, Long and Shen, Zuo-Jun Max},
  journal={Management Science},
  volume={67},
  number={7},
  pages={4095--4119},
  year={2021},
  publisher={INFORMS}
}

@article{kannan2025data,
  title = {Technical note---{D}ata-driven sample average approximation with covariate information},
  author = {Kannan, Rohit and Bayraksan, G{\"u}zin and Luedtke, James R.},
  journal={Operations Research},
  volume={73},
  number={6},
  pages={3245--3259},
  year={2025}
}

@article{ferreira2016analytics,
  title={Analytics for an online retailer: Demand forecasting and price optimization},
  author={Ferreira, Kris Johnson and Lee, Bin Hong Alex and Simchi-Levi, David},
  journal={Manufacturing \& Service Operations Management},
  volume={18},
  number={1},
  pages={69--88},
  year={2016},
  publisher={INFORMS}
}

@book{quinonero2008dataset,
  editor    = {Qui{\~n}onero-Candela, Joaquin and Sugiyama, Masashi and Schwaighofer, Anton and Lawrence, Neil D.},
  title     = {Dataset shift in machine learning},
  series    = {Neural Information Processing Series},
  publisher = {MIT Press},
  address   = {Cambridge, MA},
  year      = {2008},
  isbn      = {978-0-262-17005-5}
}

@article{ben2013robust,
  title={Robust solutions of optimization problems affected by uncertain probabilities},
  author={Ben-Tal, Aharon and den Hertog, Dick and De Waegenaere, Anja and Melenberg, Bertrand and Rennen, Gijs},
  journal={Management Science},
  volume={59},
  number={2},
  pages={341--357},
  year={2013},
  publisher={INFORMS}
}

@article{bertsimas2018data,
  title={Data-driven robust optimization},
  author={Bertsimas, Dimitris and Gupta, Vishal and Kallus, Nathan},
  journal={Mathematical Programming},
  volume={167},
  number={2},
  pages={235--292},
  year={2018},
  publisher={Springer}
}

@phdthesis{dunn2018,
  title   = {Optimal trees for prediction and prescription},
  author  = {Dunn, Jack William},
  school  = {Massachusetts Institute of Technology},
  address = {Cambridge, MA},
  year    = {2018}
}

@article{elmachtoub2022smart,
  title     = {Smart ``predict, then optimize''},
  author    = {Elmachtoub, Adam N. and Grigas, Paul},
  journal   = {Management Science},
  volume    = {68},
  number    = {1},
  pages     = {9--26},
  year      = {2022},
  publisher = {INFORMS}
}

@article{bertsimas2017optimal,
  title={Optimal classification trees},
  author={Bertsimas, Dimitris and Dunn, Jack},
  journal={Machine Learning},
  volume={106},
  number={7},
  pages={1039--1082},
  year={2017},
  publisher={Springer}
}

@inproceedings{wilder2019melding,
  title     = {Melding the data-decisions pipeline: Decision-focused learning for combinatorial optimization},
  author    = {Wilder, Bryan and Dilkina, Bistra and Tambe, Milind},
  booktitle = {Proceedings of the AAAI Conference on Artificial Intelligence},
  volume    = {33},
  pages     = {1658--1665},
  year      = {2019}
}

@inproceedings{donti2017task,
 author = {Donti, Priya and Amos, Brandon and Kolter, J. Zico},
 booktitle = {Advances in Neural Information Processing Systems},
 pages = {5484--5494},
 title = {Task-based end-to-end model learning in stochastic optimization},
 volume = {30},
 year = {2017}
}

@inproceedings{elmachtoub2020decision,
  title     = {Decision trees for decision-making under the predict-then-optimize framework},
  author    = {Elmachtoub, Adam N. and Liang, Jason Cheuk Nam and McNellis, Ryan},
  booktitle = {Proceedings of the 37th International Conference on Machine Learning},
  series    = {Proceedings of Machine Learning Research},
  volume    = {119},
  pages     = {2858--2867},
  publisher = {PMLR},
  year      = {2020}
}

@article{ban2019big,
  title={The big data newsvendor: Practical insights from machine learning},
  author={Ban, Gah-Yi and Rudin, Cynthia},
  journal={Operations Research},
  volume={67},
  number={1},
  pages={90--108},
  year={2019},
  publisher={INFORMS}
}

@article{van2021data,
  title={From data to decisions: Distributionally robust optimization is optimal},
  author={Van Parys, Bart P. G. and Mohajerin Esfahani, Peyman and Kuhn, Daniel},
  journal={Management Science},
  volume={67},
  number={6},
  pages={3387--3402},
  year={2021},
  publisher={INFORMS}
}

@article{goerigk2023data,
  title={Data-driven robust optimization using deep neural networks},
  author={Goerigk, Marc and Kurtz, Jannis},
  journal={Computers \& Operations Research},
  volume={151},
  pages={106087},
  year={2023},
  publisher={Elsevier}
}

@article{sadana2025survey,
  title={A survey of contextual optimization methods for decision-making under uncertainty},
  author={Sadana, Utsav and Chenreddy, Abhilash and Delage, Erick and Forel, Alexandre and Frejinger, Emma and Vidal, Thibaut},
  journal={European Journal of Operational Research},
  volume={320},
  number={2},
  pages={271--289},
  year={2025},
  publisher={Elsevier}
}


\end{document}